\documentclass[a4paper,10pt,reqno]{amsart}
\usepackage{amsmath,amsfonts,amscd,amssymb,graphicx,mathrsfs,eufrak,dsfont}
\usepackage[all,2cell]{xy}
\usepackage[usenames]{color}

\usepackage[colorlinks]{hyperref}
\usepackage{tikz,mathrsfs}
\usetikzlibrary{arrows,decorations.pathmorphing,decorations.pathreplacing,positioning,shapes.geometric,shapes.misc,decorations.markings,decorations.fractals,calc,patterns}

\tikzset{>=stealth',
         dvertex/.style={circle,draw=white,inner sep=5pt,outer sep=7pt},
         vertex/.style={circle,fill=black,inner sep=1pt,outer sep=3pt},
         star/.style={circle,fill=yellow,inner sep=0.75pt,outer sep=0.75pt},
         tvertex/.style={inner sep=1pt,font=\scriptsize},
         gap/.style={inner sep=0.5pt,fill=white}}
\tikzstyle{mybox} = [draw=black, fill=blue!10, very thick,
    rectangle, rounded corners, inner sep=10pt, inner ysep=20pt]
\tikzstyle{boxtitle} =[fill=blue!50, text=white,rectangle,rounded corners]

\newtheorem{thm}{Theorem}[section]
\newtheorem{prop}[thm]{Proposition}
\newtheorem{lemma}[thm]{Lemma}
\newtheorem{defin}[thm]{Definition}
\newtheorem{cor}[thm]{Corollary}

\theoremstyle{definition} 
 
\newtheorem{example}[thm]{Example}

\newtheorem{remark}[thm]{Remark}

\numberwithin{equation}{section}

\def\thick{\mathop{\rm thick}\nolimits}

\newcommand{\e}{\varepsilon}

\renewcommand{\t}[1]{\textnormal{#1}}

\newcommand{\looptop}[2]{\xy \SelectTips{cm}{10}
\POS(0,0) \endxy}

\def\Sub{\mathop{\rm Sub}\nolimits}
\def\Aus{\mathop{\rm Aus}\nolimits}

\def\op{\mathop{\rm op}\nolimits}
\def\SL{\mathop{\rm SL}\nolimits}
\def\GL{\mathop{\rm GL}\nolimits}
\def\CM{\mathop{\rm CM}\nolimits}
\def\GP{\mathop{\rm GP}\nolimits}
\def\uCM{\mathop{\underline{\rm CM}}\nolimits}

\def\SCM{\mathop{\rm SCM}\nolimits}

\def\depth{\mathop{\rm depth}\nolimits}

\def\mod{\mathop{\rm mod}\nolimits}

\def\top{\mathop{\rm top}\nolimits}
\def\coh{\mathop{\rm coh}\nolimits}
\def\Qcoh{\mathop{\rm Qcoh}\nolimits}
\def\Mod{\mathop{\rm Mod}\nolimits}

\def\proj{\mathop{\rm proj}\nolimits}
\def\Inj{\mathop{\rm Inj}\nolimits}

\def\pd{\mathop{\rm proj.dim}\nolimits}
\def\id{\mathop{\rm inj.dim}\nolimits}
\def\Hom{\mathop{\rm Hom}\nolimits}
\def\RHom{\mathop{\rm {\bf R}Hom}\nolimits}
\def\End{\mathop{\rm End}\nolimits}
\def\Ext{\mathop{\rm Ext}\nolimits}

\def\add{\mathop{\rm add}\nolimits}

\def\ker{\mathop{\rm ker}\nolimits}

\def\Sing{\mathop{\rm Sing}\nolimits}

\def\Spec{\mathop{\rm Spec}\nolimits}
\def\Max{\mathop{\rm Max}\nolimits}

\def\gl{\mathop{\rm gl.dim}\nolimits}

\def\D{\mathop{\rm{D}^{}}\nolimits}
\def\Dsg{\mathop{\rm{D}_{\rm sg}}\nolimits}
\def\Db{\mathop{\rm{D}^b}\nolimits}
\def\Kb{\mathop{\rm{K}^b}\nolimits}
\def\K{\mathop{\rm{K}}\nolimits}
\def\m{\mathop{\mathfrak{m}}\nolimits}
\def\Rf{\mathop{\mathbf{R}f_*}\nolimits}
\def\Rfcs{\mathop{\mathbf{R}f_*^\cs}\nolimits}
\def\Lf{\mathop{\mathbf{L}f^*}\nolimits}
\def\fsh{\mathop{f^!}\nolimits}
\def\gsh{\mathop{g^!}\nolimits}
\def\Perf{\mathop{\rm{Perf}}\nolimits}

\def\RsHom{\mathop{{\bf R}\mathcal{H}om}\nolimits}

\newcommand{\ul}[1]{\underline{#1}}

\newcommand{\ra}{\to}

\newcommand{\G}{\mathbb{G}}
\newcommand{\I}{\mathbb{I}}
\renewcommand{\P}{\mathbb{P}}

\newcommand{\cc}{{\mathcal C}}
\newcommand{\cd}{{\mathcal D}}
\newcommand{\ce}{{\mathcal E}}
\newcommand{\cf}{{\mathcal F}}
\newcommand{\cg}{{\mathcal G}}

\newcommand{\ci}{{\mathcal I}}

\newcommand{\cl}{{\mathcal L}}
\newcommand{\cm}{{\mathcal M}}
\newcommand{\cn}{{\mathcal N}}
\newcommand{\co}{{\mathcal O}}
\newcommand{\cp}{{\mathcal P}}
\newcommand{\cq}{{\mathcal Q}}
\newcommand{\cs}{{\mathcal S}}
\newcommand{\ct}{{\mathcal T}}
\newcommand{\cu}{{\mathcal U}}
\newcommand{\cv}{{\mathcal V}}

\newcommand{\cz}{{\mathcal Z}}

\begin{document}
\title[Frobenius Categories via Gorenstein Algebras]{Frobenius categories, Gorenstein algebras and rational surface singularities}
\author{Osamu Iyama}
\address{Osamu Iyama\\ Graduate School of Mathematics\\ Nagoya University\\ Chikusa-ku, Nagoya, 464-8602, Japan}
\email{iyama@math.nagoya-u.ac.jp}
\author{Martin Kalck}
\address{Martin Kalck, The Maxwell Institute, School of Mathematics, James Clerk Maxwell Building, The King's Buildings, Mayfield Road, Edinburgh, EH9 3JZ, UK.}
\email{m.kalck@ed.ac.uk}
\author{Michael Wemyss}
\address{Michael Wemyss, The Maxwell Institute, School of Mathematics, James Clerk Maxwell Building, The King's Buildings, Mayfield Road, Edinburgh, EH9 3JZ, UK.}
\email{wemyss.m@googlemail.com}
\author{Dong Yang}
\address{Dong Yang, Department of Mathematics, Nanjing University, Nanjing 210093, P. R. China}
\email{dongyang2002@gmail.com}
\dedicatory{Dedicated to Ragnar-Olaf Buchweitz on the occasion of his 60th birthday.}
\thanks{\emph{2010 Mathematical Subject Classification:} 14J17, 13C14, 18E30, 16E65}
\thanks{\emph{Keywords:} Frobenius categories,
Iwanaga--Gorenstein algebras, Gorenstein-projective modules,
Cohen-Macaulay modules, 
noncommutative resolutions,
singularity categories,
rational surface singularities.}
\begin{abstract}
We give sufficient conditions for a Frobenius category to be equivalent to the category of Gorenstein projective modules over an Iwanaga--Gorenstein ring.\  We then apply this result to the Frobenius category of special Cohen--Macaulay modules over a rational surface singularity, where we show that the associated stable category is triangle equivalent to the singularity category of a certain discrepant partial resolution of the given rational singularity.
In particular, this produces uncountably many Iwanaga--Gorenstein rings of finite GP type. We also apply our method to representation theory, obtaining Auslander--Solberg and Kong type results.
\end{abstract}
\maketitle

\tableofcontents

\parindent 15pt

\section{Introduction} This paper is motivated by the study of certain triangulated categories associated to rational surface singularities, first constructed in \cite{IW2}.  The purpose is to develop both the algebraic and geometric techniques necessary to give precise information regarding these categories, and to put them into a more conceptual framework.  It is only by developing both sides of the picture that we are able to prove the results that we want.

We explain the algebraic side first.  Frobenius categories \cite{H88,K,Keller96} are now ubiquitous in algebra, since they give rise to many of the triangulated categories arising in algebraic and geometric contexts.  One of the points of this paper is that we should treat Frobenius categories which admit a `noncommutative resolution' as a special class of Frobenius categories. We show that such a Frobenius category is equivalent to the category ${\GP}(E)$ of Gorenstein projective modules over some Iwanaga--Gorenstein ring $E$ (for definitions see \S\ref{ss:alg-tria-cl-sing} and \S\ref{Frob section}).   The precise statement is as follows.  
For a Frobenius category $\ce$ we denote by $\proj\ce$ the full subcategory of $\ce$ consisting of projective objects and for an object $P$ of $\ce$ we denote by $\add P$ the full subcategory of $\ce$ consisting of direct summands of finite direct sums of copies of $P$. 

\begin{thm}(=\ref{t:main-thm})\label{main Frob intro}
Let $\ce$ be a Frobenius category with $\proj \ce= \add P$ for some $P \in \proj \ce$.  Assume that there exists $M\in\ce$ such that $A:=\End_{\ce}(P\oplus M)$ is a noetherian ring of global dimension $n$.  Then\\ 
\t{(1)} $E:=\End_\ce(P)$ is an Iwanaga--Gorenstein ring of dimension at most $n$, that is, a noetherian ring with  $\id_E E\leq n$ and $\id E_E\leq n$.\\
\t{(2)} We have an equivalence $\Hom_{\ce}(P,-)\colon\ce\to\GP(E)$ up to direct summands. It is an equivalence if $\ce$ is idempotent complete. This induces a triangle equivalence 
\[
\underline{\ce}\xrightarrow{\simeq}\underline{\GP}(E)\simeq \Dsg(E)
\]
up to direct summands. It is an equivalence if $\ce$ or $\underline{\ce}$ is idempotent complete.\\
\t{(3)} $\ul{\ce}=\thick_{\ul{\ce}}(M)$, i.e.\ the smallest full triangulated subcategory of $\ul{\ce}$ containing $M$ which is closed under direct summands is $\ul{\ce}$.
\end{thm}

This abstract result has applications in, and is motivated by, problems in algebraic geometry. If $R$ is a Gorenstein singularity, then the category $\CM(R)$ of maximal Cohen--Macaulay modules over $R$ is a Frobenius category.  Moreover if $R$ is a simple surface singularity, then the classical algebraic McKay correspondence can be formulated in terms of the associated stable category $\uCM(R)$, see \cite{Auslander86}.

When $R$ is not Gorenstein, $\CM(R)$ is no longer Frobenius.  However, for a complete local rational surface singularity $R$ over an algebraically closed field of characteristic zero (for details, see \S\ref{SCMsection} or \S\ref{uncount section}), there is a subcategory $\SCM(R)\subseteq\CM(R)$ of \emph{special} CM modules (recalled in \S\ref{SCMsection}). By Wunram's GL(2) McKay correspondence \cite{Wun}, if we denote $Y\to\Spec R$ to be the minimal resolution, and let $\{E_i\}_{i\in I}$ denote the set of exceptional curves, then there is a natural bijection
\[
\left. \left\{ \begin{array}{c}\mbox{non-free indecomposable}\\ \mbox{special CM $R$-modules}\end{array}\right\}\middle/ \cong \right.\quad\longleftrightarrow\quad\left\{ E_i\mid i\in I\right\}.
\]
We let $M_i$ denote the indecomposable special CM $R$-module corresponding to the exceptional curve $E_i$.  We remark that the set of exceptional curves can be partitioned into two subsets, namely $I=\mathcal{C}\cup \mathcal{D}$ where $\mathcal{C}$ are all the (C)repant  curves (i.e.\ the ($-2$)-curves), and $\mathcal{D}$ are all the (D)iscrepant curves (i.e.\ the non-($-2$)-curves).  In this paper, the following module plays a central role.

\begin{defin}\label{module N}
We define the module $D\in\SCM(R)$ by $D:=R\oplus\left(\bigoplus_{d\in\mathcal{D}}M_d\right)$.
\end{defin}

It was shown in \cite{IW2} that the category $\SCM(R)$ has at least one natural Frobenius structure.  Our first result in this setting is that there are often many different Frobenius structures on $\SCM(R)$, and so the one found in \cite{IW2} is not unique.

\begin{prop}(=\ref{new Frobenius})\label{new Frobenius intro}
Let R be a complete local rational surface singularity over an algebraically closed field of characteristic zero and let $D\in\SCM(R)$ be defined as above.  Choose $N\in\SCM(R)$ such that $\add D\subseteq\add N$. Then $\SCM(R)$ has the structure of a Frobenius category whose projective objects are exactly $\add N$. We denote the category $\SCM(R)$, equipped with this Frobenius structure, by $\SCM_N(R)$.
\end{prop}

It then follows from \ref{main Frob intro} and \ref{new Frobenius intro} that $\End_R(N)$ is Iwanaga--Gorenstein (\ref{main}),  using the fact that the reconstruction algebra, i.e. $\End_R(R\oplus\bigoplus_{i\in I}M_i)$, has finite global dimension \cite{IW1,GL2}. 

We then interpret the stable category $\underline{\SCM}_N(R)$ of the Frobenius category $\SCM_N(R)$ geometrically.   To do this, we remark that the condition $\add D\subseteq\add N$ implies (after passing to the basic module) that we can write
\[
N=D\oplus \bigoplus_{j\in J}M_j
\] 
for some subset $J\subseteq \mathcal{C}$.  Set $\mathcal{S}:=\mathcal{C}\backslash J$, the complement of $J$ in $\mathcal{C}$, so that
\[
N=D\oplus \bigoplus_{j\in \mathcal{C}\backslash \mathcal{S}}M_j:=N^{\mathcal{S}}.
\]
Contracting all the curves in $\cs$, we obtain a scheme $X^{\cs}$ together with maps
\[
Y\xrightarrow{f^{\cs}} X^{\cs}\xrightarrow{g^{\cs}}\Spec R.
\]

Knowledge of the derived category of $X^\cs$ leads to our main result, which also explains geometrically why $\End_R(N^\cs)$ is Iwanaga--Gorenstein (\ref{Goren from geom}).

\begin{thm}(=\ref{main Db}, \ref{main triangles})\label{main Db triangles intro}  With the assumptions as in \ref{new Frobenius intro}, choose $\mathcal{S}\subseteq\mathcal{C}$ (i.e.\ $N^{\cs}\in\SCM(R)$ such that $\add D\subseteq\add N^\cs$).  Then\\
\t{(1)} There is a derived equivalence between $\End_R(N^\cs)$ and $X^\cs$.\\
\t{(2)} As a consequence, we obtain triangle equivalences
\[
\underline{\SCM}_{N^{\cs}}(R)\simeq\underline{\GP}(\End_R(N^{\cs}))\simeq\Dsg(\End_R(N^{\cs}))\simeq\Dsg(X^{\cs})\simeq\bigoplus_{x\in\Sing X^{\cs}}\uCM(\widehat{\mathcal{O}}_{X^{\cs}, x})
\]
where $\Sing X^{\mathcal{S}}$ denotes the set of singular points of $X^{\mathcal{S}}$. \\
\t{(3)} In particular, $\underline{\SCM}_{N^{\mathcal{S}}}(R)$ is 1-Calabi--Yau, and its shift functor satisfies $[2]=\mathrm{id}$.
\end{thm}

Thus \ref{main Db triangles intro} shows that $\underline{\SCM}_{N^{\cs}}(R)$ is nothing other than the usual singularity category of some partial resolution of $\Spec R$.  We remark that it is the geometry that determines the last few statements in \ref{main Db triangles intro}, as we are unable to prove them using algebra alone. In \S\ref{ss:relative-sing-cat} we give a relative singularity category version of the last two equivalences in \ref{main Db triangles intro}.

The following corollary to \ref{main Db triangles intro} extends \cite[4.11]{IW2} and gives a `relative' version of Auslander's algebraic McKay correspondence for all rational surface singularities.

\begin{cor}(=\ref{McKay rel})
With the assumptions as in \ref{new Frobenius intro}, choose $N^{\cs}\in\SCM(R)$ such that $\add D\subseteq\add N^\cs$.  Then the AR quiver of the category $\underline{\SCM}_{N^{\cs}}(R)$ is the double of the dual graph with respect to the morphism $Y\to X^\cs$.
\end{cor}

Using the geometry, we are also able to improve \ref{main Frob intro}(1) in the situation of rational surface singularities, since we are able to give the precise value of the injective dimension. The following is a generalization of a known result \ref{glrecon3} for the case $\add N=\SCM(R)$.

\begin{thm}(=\ref{precise id})\label{precise id intro}
With the assumptions as in \ref{new Frobenius intro}, choose $N\in\SCM(R)$ such that $\add D\subseteq\add N$ and put $\Gamma=\End_R(N)$. Then 
\[
\id_\Gamma \Gamma =\left\{ \begin{array}{cl} 2 & \mbox{if $R$ is Gorenstein}\\ 3 & \mbox{else.}\end{array}  \right. 
\]
\end{thm}

This gives many new examples of Iwanaga--Gorenstein rings $\Gamma$, of injective dimension three, for which there are only finitely many Gorenstein--projective modules up to isomorphism.  In contrast to the commutative situation, we also show the following result. Some explicit examples are given in \S\ref{geom examples}. 

\begin{thm}\label{uncount intro}(=\ref{uncount G})
Let $G\leq\SL(2,\mathbb{C})$ be a finite subgroup, with $G\ncong E_8$.  Then there are uncountably many non-isomorphic Iwanaga--Gorenstein rings $\Lambda$ with $\id_\Lambda\Lambda=3$, such that $\underline{\GP}(\Lambda)\simeq\uCM(\mathbb{C}[[x,y]]^{G})$.
\end{thm}

\medskip
\noindent\emph{Conventions and notations.} We use the convention that the composition of morphisms $f:X\rightarrow Y$ and $g:Y\rightarrow Z$ in a category is denoted by $fg$. By a module over a ring $A$ we mean a left module, and we denote by $\Mod A$ (resp.\ $\mod A$) the category of $A$-modules (resp.\ finitely generated $A$-modules).  We denote by $\proj A$ the category of finitely generated projective $A$-modules.  If $M$ is an object of an additive category $\mathcal{C}$, we denote by $\add M$ all those objects of $\mathcal{C}$ which are direct summands of (finite) direct sums of $M$.  We say that $M$ is an \emph{additive generator} of $\mathcal{C}$ if $\mathcal{C}=\add M$.   If $\ct$ is a triangulated category and $M\in\ct$, we denote by $\thick(M)$ the smallest full triangulated subcategory  containing $M$ which is closed under taking direct summands.

\section{A Morita type theorem for Frobenius categories} \label{ss:alg-tria-cl-sing}

Throughout this section let $\ce$ denote a Frobenius category, and denote by $\proj \ce\subseteq \ce$ the full subcategory of projective objects.  We denote the stable category of $\ce$ by $\ul{\ce}$. It has the same objects as $\ce$, but the morphism spaces are defined as $\ul{\Hom}_{\ce}(X, Y)=\Hom_{\ce}(X, Y)/\cp(X, Y)$, where $\cp(X, Y)$ is the subspace of morphisms factoring through $\proj \ce$. We refer to Keller's overview article for definitions and unexplained terminologies \cite{K,Keller96}.

\subsection{Frobenius categories as categories of Gorenstein projective modules}\label{Frob section}

Recall that a noetherian ring $E$ is called \emph{Iwanaga--Gorenstein of dimension at most $n$} if $\id_EE\le n$ and $\id E_E\le n$ \cite{EJ}. 
For an Iwanaga--Gorenstein ring $E$ of dimension at most $n$, we denote by
\[
\GP(E):=\{X\in\mod E\mid \Ext^i_E(X,E)=0\mbox{ for any }i>0\}=\Omega^n(\mod E),
\]
the category of \emph{Gorenstein projective} $E$-modules \cite{AB, EJ}. Here $\Omega$ is the syzygy functor of $\mod E$.
This is a Frobenius category with $\proj E$ the subcategory of projective objects.

\begin{remark}
The objects of $\GP(E)$ are sometimes called Cohen--Macaulay modules, but there are reasons why we do not do this; see \ref{CM not CM} later. They are sometimes also called totally reflexive modules.
\end{remark}

\begin{defin}
\cite{Buch, Orlov}
Let $R$ be a left noetherian ring. The triangulated category $\Dsg(R):=\Db(\mod R)/\Kb(\proj R)$ is called the \emph{singularity category} of $R$.
\end{defin}

\begin{remark}\label{Buchweitz}
Let $E$ be an Iwanaga--Gorenstein ring.
By a result of Buchweitz~\cite[4.4.1 (2)]{Buch}, we have an equivalence of triangulated categories
\[
\ul{\GP}(E)\simeq \Dsg(E).
\]
\end{remark}

The purpose of this section is to show that the existence of a noncommutative resolution of a Frobenius category $\ce$ puts strong restrictions on $\ce$ (\ref{t:main-thm}). 

\begin{defin}\label{defNCR}
Let $\ce$ be a Frobenius category with $\proj \ce= \add P$ for some $P \in \proj \ce$.  By a \emph{noncommutative resolution} of $\ce$, we mean $A:=\End_{\ce}(M)$ for some $M\in\ce$ with $P\in\add M$, such that $A$ is noetherian with $\gl A<\infty$.
\end{defin}

In the level of generality of abstract Frobenius categories, the above definition is new.  We remark that when $\ce=\CM R$ (with $R$ a Gorenstein ring), our definition of noncommutative resolution is much weaker than Van den Bergh's notion of a noncommutative crepant resolution (=NCCR) \cite{VdBNCCR}, and especially in higher dimension, examples occur much more often.

\begin{remark}
Not every Frobenius category with a projective generator admits a noncommutative resolution.  Indeed, let $R$ be a normal Gorenstein surface singularity, over $\mathbb{C}$, and consider $\ce:=\CM(R)$.   Then any noncommutative resolution in the above sense is automatically an NCCR, and the existence of an NCCR is well-known to imply that $R$ must have rational singularities \cite{SVdB}.
\end{remark}

Our strategy to prove \ref{t:main-thm} is based on \cite[2.2(a)]{AIR}, but the setup here is somewhat different. We need the following technical observation.

\begin{lemma}\label{technical lemma}
Let $\ce$ be a Frobenius category with $\proj \ce= \add P$ for some $P \in \proj \ce$.  If $f:X\to Y$ is a morphism in $\ce$ such that $\Hom_{\ce}(f,P)$ is surjective, then there exists an exact sequence
\[
0\to X\xrightarrow{(f\ 0)}Y\oplus P'\to Z\to0
\]
in $\ce$ with $P'\in\proj \ce$.
\end{lemma}
\begin{proof}
This follows for example from \cite[2.10]{Kalck}. 
\end{proof}

\begin{thm}\label{t:main-thm}
Let $\ce$ be a Frobenius category with $\proj \ce= \add P$ for some $P \in \proj \ce$. Assume that there exists a noncommutative resolution $\End_{\ce}(M)$ of $\ce$ with $\gl\End_{\ce}(M)=n$. Then\\ 
\t{(1)} $E:=\End_\ce(P)$ is an Iwanaga--Gorenstein ring of dimension at most $n$.\\
\t{(2)} We have an equivalence $\Hom_{\ce}(P,-)\colon\ce\to\GP(E)$ up to direct summands. It is an equivalence if $\ce$ is idempotent complete. This induces a triangle equivalence 
\[
\underline{\ce}\xrightarrow{\simeq}\underline{\GP}(E)\simeq \Dsg(E)
\]
up to direct summands. It is an equivalence if $\ce$ or $\underline{\ce}$ is idempotent complete.\\
\t{(3)} $\ul{\ce}=\thick_{\ul{\ce}}(M)$.
\end{thm}
\begin{proof}
Since $P\in\add M$, $\End_{\ce}(M)$ is Morita equivalent to $A:=\End_{\ce}(P\oplus M)$ and so $\gl A=n$. Since $A$ is noetherian, so is $E$ (see for example \cite[1.1.7]{MR}).  It follows from a standard argument that the functor $\Hom_{\ce}(P,-):\ce\to\mod E$ is fully faithful, restricting to an equivalence $\Hom_{\ce}(P,-):\add P\to\proj E$ up to direct summands.  We can drop the `up to direct summands' assumption if $\ce$ is idempotent complete. We establish (1) in three steps:

(i) We first show that $\Ext^i_E(\Hom_{\ce}(P,X),E)=0$ for any $X\in\ce$ and $i>0$.  Let
\begin{equation}\label{P resolution}
0\to Y\to P'\to X\to 0
\end{equation}
be an exact sequence in $\ce$ with $P'$ projective.
Applying $\Hom_{\ce}(P,-)$, we have an exact sequence
\begin{equation}\label{4th syzygy}
0\to\Hom_{\ce}(P,Y)\to\Hom_{\ce}(P,P')\to\Hom_{\ce}(P,X)\to0
\end{equation}
with a projective $E$-module $\Hom_{\ce}(P,P')$.
Applying $\Hom_{\ce}(-,P)$ to \eqref{P resolution} and $\Hom_E(-,E)$ to \eqref{4th syzygy}
respectively and comparing them, 
we have a commutative diagram of exact sequences
\[
\begin{array}{c}
{\SelectTips{cm}{10}
\xy0;/r.3pc/:
(-15,20)*+{\Hom_\ce(P^\prime,P)}="A0",(18,20)*+{\Hom_\ce(Y,P)}="A1",(50,20)*+{0}="A2",
(-15,10)*+{\Hom_E(\Hom_{\ce}(P,P'),E)}="a0",(18,10)*+{\Hom_E(\Hom_{\ce}(P,Y),E)}="a1",(50,10)*+{\Ext^1_E(\Hom_{\ce}(P,X),E)}="a2",(68,10)*+{0}="a3",
\ar"A0";"A1"
\ar"A1";"A2"
\ar"a0";"a1"
\ar"a1";"a2"
\ar"a2";"a3"
\ar"A0";"a0"^\wr
\ar"A1";"a1"^\wr
\endxy}
\end{array}
\]
Thus we have $\Ext^1_E(\Hom_{\ce}(P,X),E)=0$.
Since the syzygy of $\Hom_{\ce}(P,X)$ has the same form $\Hom_{\ce}(P,Y)$, we have $\Ext^i_E(\Hom_{\ce}(P,X),E)=0$ for any $i>0$.

(ii) We show that for any $X\in\mod E$, there exists an exact sequence
\begin{equation}\label{approximation}
0\to Q_n\to\cdots\to Q_0\to X\to0
\end{equation}
of $E$-modules with $Q_i\in\add\Hom_{\ce}(P,P\oplus M)$.

Define an $A$-module by $\widetilde{X}:=\Hom_{\ce}(P\oplus M,P)\otimes_EX$.
Let $e$ be the idempotent of $A=\End_{\ce}(P\oplus M)$ corresponding to the direct summand $P$ of $P\oplus M$.
Then we have $eAe=E$ and $e\widetilde{X}=X$.
Since the global dimension of $A$ is at most $n$, there exists a projective resolution
\[
0\to P_n\to\cdots\to P_0\to\widetilde{X}\to0.
\]
Applying $e(-)$ and using $eA=\Hom_{\ce}(P,P\oplus M)$, we have the assertion.

(iii) By (i) and (ii), we have that $\Ext^{n+1}_E(X,E)=0$ for any $X\in\mod E$,
and so the injective dimension of the $E$-module $E$ is at most $n$.
The dual argument shows that the injective dimension of the $E^{\rm op}$-module $E$ is at most $n$.
Thus $E$ is Iwanaga--Gorenstein, which shows (1).\\
(2) By (i) again, we have a functor $\Hom_{\ce}(P,-):\ce\to\GP(E)$, and it is fully faithful.  We will now show that it is dense up to direct summands.  

For any $X\in\GP(E)$, we take an exact sequence \eqref{approximation}.
Since $Q_i\in\add\Hom_{\ce}(P,P\oplus M)$, we have a complex 
\begin{equation}\label{add P+M}
M_n\xrightarrow{f_n}\cdots\xrightarrow{f_0} M_0
\end{equation}
in $\ce$ with $M_i\in\add(P\oplus M)$ such that
\begin{equation}\label{(P,P+M)}
0\to\Hom_{\ce}(P,M_n)\xrightarrow{\cdot f_n}\hdots\xrightarrow{\cdot f_0}\Hom_{\ce}(P,M_0)\to X\oplus Y\to0
\end{equation}
is exact for some $Y\in\GP(E)$. (Note that due to the possible lack of direct summands in $\ce$ it is not always possible to choose $M_i$ such that $\Hom_{\ce}(P,M_i)=Q_i$.)  Applying $\Hom_{\ce}(-,P)$ to \eqref{add P+M} and $\Hom_E(-,E)$ to  \eqref{(P,P+M)} and comparing them, we have a commutative diagram
\[
\begin{xy}
\SelectTips{cm}{}
\xymatrix{
\Hom_{\ce}(M_0,P)\ar[r]\ar[d]^\wr&\cdots\ar[r]&\Hom_{\ce}(M_n,P)\ar[d]^\wr\ar[r]&0\\
\Hom_E(\Hom_{\ce}(P,M_0),E)\ar[r]&\cdots\ar[r]&\Hom_E(\Hom_{\ce}(P,M_n),E)\ar[r]&0
}
\end{xy}
\]
where the lower sequence is exact since $X\oplus Y\in\GP(E)$.
Thus the upper sequence is also exact.  But applying \ref{technical lemma} repeatedly to \eqref{add P+M}, we have
a complex
\[0\to M_n\xrightarrow{(f_n\ 0)} M_{n-1}\oplus P_{n-1}\xrightarrow{{f_{n-1}\ 0\ 0\choose \ \ 0\ \ \ 1\ 0}}
M_{n-2}\oplus P_{n-1}\oplus P_{n-2}\xrightarrow{}
\cdots\to M_0\oplus P_1\oplus P_0\to N\to0\]
with projective objects $P_i$ which is a glueing of exact sequences in $\ce$.
Then we have $X\oplus Y\oplus \Hom_{\ce}(P,P_0)\simeq \Hom_{\ce}(P,N)$, and we have the assertion. The final statement follows by \ref{Buchweitz}.\\
(3) The existence of (\ref{approximation}) implies that $\Hom_{\ce}(P, -)$
gives a triangle equivalence $\thick_{\underline{\ce}}(M)\to\underline{\GP}(E)$
up to direct summands. Thus the natural inclusion $\thick_{\underline{\ce}}(M)\to
\underline{\ce}$ is also a triangle equivalence up to direct summands.
This must be an isomorphism since $\thick_{\underline{\ce}}(M)$ is closed under
direct summands in $\underline{\ce}$.
\end{proof}

We note the following more general version stated in terms of functor categories \cite{Auslander66}. For an additive category $\cp$ we denote by $\Mod\cp$ the category of contravariant additive functors from $\cp$ to the category of abelian groups.
For $X\in\ce$, we have a $\cp$-module $H_X:=\Hom_{\ce}(-,X)|_{\cp}$.
We denote by $\mod\cp$ the full subcategory of $\Mod\cp$ consisting of finitely presented objects.
Similarly we define $\Mod\cp^{\rm op}$, $H^X$ and $\mod\cp^{\rm op}$. If $\cp$ has pseudokernels (respectively, pseudocokernels), then $\mod\cp$ (respectively, $\mod\cp^{\rm op}$) is an abelian category.

\begin{thm}\label{main-thm-category-version}
Let $\ce$ be a Frobenius category with the category $\cp$ of projective objects.
Assume that there exists a full subcategory $\cm$ of $\ce$ such that $\cm$ contains $\cp$, $\cm$ has pseudokernels and pseudocokernels and $\mod\cm$ and $\mod\cm^{\rm op}$ have global dimension at most $n$. Then\\
\t{(1)} $\cp$ is an Iwanaga--Gorenstein category of dimension at most $n$, i.e.\ $\Ext^i_{\mod\cp}(-,H_P)=0$ and $\Ext^i_{\mod\cp^{\rm op}}(-,H^P)=0$ for all $P\in\cp$, $i>n$.\\
\t{(2)} For the category
\[
\GP(\cp):=\{X\in\mod\cp\mid \Ext^i_{\cp}(X,H_P)=0\mbox{ for any }i>0\mbox{ and }P\in\cp\}
\]
of \emph{Gorenstein projective} $\cp$-modules, we have an equivalence $\ce\to\GP(\cp)$, $X\mapsto H_X$ up to  summands.  It is an equivalence if $\ce$ is idempotent complete. This induces a triangle equivalence
\[
\underline{\ce}\to\underline{\GP}(\cp)\simeq \Dsg(\cp).
\]
up to  summands.  It is an equivalence if $\ce$ or $\underline{\ce}$ is idempotent complete.\\
\t{(3)} $\ul{\ce}=\thick_{\ul{\ce}}(M)$.
\end{thm}

\begin{remark}
In the setting of \ref{main-thm-category-version}, we remark that  \cite[4.2]{C} also gives an embedding $\ce\to\GP(\cp)$.
\end{remark}

\subsection{Alternative Approach}

We now give an alternative proof of \ref{t:main-thm} by using certain quotients of derived categories.  This will be necessary to interpret some results in \S\ref{ss:relative-sing-cat} later.  We retain the setup from the previous subsection, in particular $\ce$ always denotes a Frobenius category.  Recall the following.
\begin{defin} 
Let $N \in \mathbb{Z}$. A complex $P^*$ of projective objects in $\ce$ is called \emph{acyclic in degrees $\leq N$} if there exist exact sequences in $\ce$
\[
\begin{xy}\SelectTips{cm}{}
\xymatrix{
Z^n(P^*) \ar@{>->}[r]^(.65){i_{n}} & P^n \ar@{->>}[r]^(.35){p_{n}} & Z^{n+1}(P^*) }
\end{xy}
\]
 such that $d^n_{P^*}=p_ni_{n+1}$ holds for all $n\leq N$. Let $\K^{-,{\rm b}}(\proj \ce) \subseteq \K^-(\proj \ce)$ be the full subcategory consisting of those complexes which are acyclic in degrees $\leq d$ for some $d \in \mathbb{Z}$. This defines a triangulated subcategory of $\K^-(\proj \ce)$ (c.f.\ \cite{KV}).
\end{defin}

Taking projective resolutions yields a functor $\ce \to \K^{-,{\rm b}}(\proj \ce)$. We need the following dual version of \cite[2.3]{KV}, see also \cite[2.36]{Kalck}. 

\begin{prop}\label{P:StFrob}
This functor induces an equivalence of triangulated categories 
\[
\P\colon \ul{\ce} \longrightarrow \K^{-,{\rm b}}(\proj \ce)/\Kb(\proj \ce). 
\]
\end{prop}

\begin{cor}\label{C:Buchweitz}
If there exists $P \in \proj \ce$ such that $\proj \ce= \add P$ and moreover $E=\End_{\ce}(P)$ is left noetherian, then there is a fully faithful triangle functor
\begin{align}
\widetilde{\P}\colon \ul{\ce} \longrightarrow  \Dsg(E).\label{def P}
\end{align}
\end{cor}
\begin{proof}
The fully faithful functor $\Hom_{\ce}(P, -)\colon \proj \ce \rightarrow \proj E$ induces a fully faithful triangle functor $\K^-(\proj \ce) \ra \K^-(\proj E)$. Its restriction $\K^{-,{\rm b}}(\proj \ce) \ra \K^{-,{\rm b}}(\proj E)$ is well defined since $P$ is projective. Define $\widetilde{\P}$ as the composition
\[
\begin{xy}\SelectTips{cm}{}
\xymatrix
{
\ul{\ce} \ar[r]^(.3){\displaystyle{\P}} & \displaystyle{\frac{\K^{-,{\rm b}}(\proj \ce)}{\Kb(\proj \ce)}} \ar[r] & \displaystyle{\frac{\K^{-,{\rm b}}(\proj E)}{\Kb(\proj E)}} \ar[r]^{\sim} & \displaystyle{\frac{\Db(\mod E)}{\Kb(\proj E)}},
}
\end{xy}
\]
where $\P$ is the equivalence from \ref{P:StFrob} and the last functor is induced by the well-known triangle equivalence $\K^{-,{\rm b}}(\proj E) \stackrel{\sim}\longrightarrow \Db(\mod E)$.
\end{proof}

\begin{remark}
In the special case when $E$ is an Iwanaga--Gorenstein ring and $\ce:=\GP(E)$, the functor $\widetilde{\P}$ in (\ref{def P}) was shown to be an equivalence in \cite[4.4.1(2)]{Buch} (see \ref{Buchweitz}).  For general Frobenius categories, $\widetilde{\P}$ is far from being an equivalence. For example, let $\ce=\proj R$ be the category of finitely generated projective modules over a left noetherian algebra $R$ equipped with the split exact structure. Then always $\ul{\ce}=0$ but $\Dsg(E)=\Dsg(R) \neq 0$ if $\gl(R)= \infty$.  We refer the reader to \cite[2.41]{Kalck} for a detailed discussion.
\end{remark}

Below in \ref{t:alternative-main}, we give a sufficient criterion for $\widetilde{\P}$  to be an equivalence.  To do this requires the following result.

\begin{prop}\label{P:OneIdempotent}
Let $A$ be a left noetherian ring and let $e \in A$ be an idempotent. The exact functor $\Hom_{A}(Ae, -)$ induces a triangle equivalence
\begin{align}\label{E:twoidempotents2}
\G\colon \frac{\Db(\mod A)/\thick(Ae)}{\thick\left(q(\mod A/AeA)\right)} \longrightarrow \frac{\Db(\mod eAe)}{\thick(eAe)},
\end{align}
where $q\colon \Db(\mod A) \to \Db(\mod A)/\thick(Ae)$ denotes the canonical projection.
\end{prop}
\begin{proof}
Taking $e=f$ in \cite[Proposition 3.3]{KY} yields the triangle equivalence \eqref{E:twoidempotents2}.  
\end{proof}

\begin{thm} \label{t:alternative-main}
Let $\ce$ be a Frobenius category with $\proj \ce= \add P$ for some $P \in \proj \ce$.  Assume that  there exists $M\in\ce$ such that $A:=\End_{\ce}(P\oplus M)$ is a left noetherian ring of finite global dimension, and denote $E:=\End_{\ce}(P)$.   Then\\ 
\t{(1)}  $\widetilde{\P}\colon \ul{\ce} \longrightarrow \Dsg(E)$ is a triangle equivalence up to direct summands. If $\ul{\ce}$ is idempotent complete, then $\widetilde{\P}$ is an equivalence. \\
\t{(2)} $\ul{\ce}=\thick_{\ul{\ce}}(M)$.
\end{thm}
\begin{proof}
Let $e \in A$ be the idempotent corresponding to the identity endomorphism $1_{P}$ of $P$, then $eAe=E$.  We have the following commutative diagram of categories and functors. 
\begin{align*}
\begin{xy}\SelectTips{cm}{}
\xymatrix{
\frac{\left( \displaystyle\Db(\mod A)/ \thick(Ae) \right)}{\displaystyle \thick\bigl(q(\mod A/AeA)\bigr)  } \ar[rr]^(.58){\displaystyle \G}_(.58){\sim} && \frac{ \displaystyle \Db(\mod eAe) }{\displaystyle \Kb(\proj eAe) } && \ul{\ce} \ar[ll]_(.39){\displaystyle \widetilde{\P}} \\
\frac{\left( \displaystyle \Kb(\proj A)/ \thick(Ae) \right)}{\displaystyle \thick \bigl(q (\mod A/AeA)\bigr)  } \ar[rr]^(.58){\displaystyle \G^{\rm restr.}}_(.52){\sim} \ar[u]^{\displaystyle \I_{1}}&& \frac{ \displaystyle \thick(eA) }{ \displaystyle \Kb(\proj eAe) } \ar[u]^{\displaystyle \I_{2}}&& \thick_{\ul{\ce}}(M) \ar[ll]_(.39){\displaystyle \widetilde{\P}^{\rm restr.}}^(.47){\sim} \ar[u]^{\displaystyle \I_{3}}
}
\end{xy}
\end{align*}
where $\I_i$ are the natural inclusions.  Since $A$ has finite global dimension the inclusion $\Kb(\proj A) \ra \Db(\mod A)$ is an equivalence and so $\I_{1}$ is an equivalence.  But $\G$ is an equivalence from \ref{P:OneIdempotent}, and $\G^{\rm restr.}$ denotes its restriction. It is also an equivalence since $\G$ maps the generator $A$ to $eA$. Thus, by commutativity of the left square we deduce that $\I_{2}$ is an equivalence. Now $\widetilde{\P}$ denotes the fully faithful functor from \ref{C:Buchweitz}, so since $\widetilde{\P}$ maps $P\oplus M$ to $\Hom_{\ce}(P, P \oplus M)$, which is isomorphic to $eA$ as left $eAe$-modules, the restriction $\widetilde{\P}^{\rm restr.}$ is a triangle equivalence up to summands. Hence the fully faithful functors $\widetilde{\P}$ and $\I_3$ are also equivalences, up to summands. In particular, $\I_3$ is an equivalence. If $\ul{\ce}$ is idempotent complete then $\thick_{\ul{\ce}}(M)$ is idempotent complete and $\widetilde{\P}^{\rm restr.}$ is an equivalence. It follows that $\widetilde{\P}$ is an equivalence in this case.
\end{proof}

\subsection{A result of Auslander--Solberg} 

Let $K$ be a field and denote $\mathbb{D}:=\Hom_K(-,K)$. The following is implicitly included in Auslander--Solberg's relative homological algebra \cite{AS93-1} (compare~\cite[5.1]{C}), and will be required later (in \S\ref{SCMsection} and \S\ref{ss:examples}) to produce examples of Frobenius categories on which we can apply our previous results.

\begin{prop}\label{new Frobenius structure}
Let $\ce$ be a $K$-linear exact category with enough projectives $\cp$ and enough injectives $\ci$. Assume that there exist an equivalence $\tau\colon\underline{\ce}\to\overline{\ce}$ and
a functorial isomorphism $\Ext^1_{\ce}(X,Y)\simeq \mathbb{D}\overline{\Hom}_{\ce}(Y,\tau X)$ for any $X,Y\in\ce$.
Let $\cm$ be a functorially finite subcategory of $\ce$ containing $\cp$ and $\ci$, and satisfies $\tau\underline{\cm}=\overline{\cm}$. Then\\
\t{(1)} Let $0\to X\xrightarrow{f} Y\xrightarrow{g} Z\to 0$ be an exact sequence in $\ce$. Then $\Hom_{\ce}(\cm,g)$ is surjective if and only if $\Hom_{\ce}(f,\cm)$ is surjective.\\
\t{(2)} $\ce$ has the structure of a Frobenius category whose projective objects are exactly $\add\cm$. More precisely, the short exact sequences of this Frobenius structure are the short exact sequences $0\to X\xrightarrow{f} Y\xrightarrow{g} Z\to 0$ of $\ce$ such that $\Hom_{\ce}(f,\cm)$ is surjective.
\end{prop}

\begin{proof}
(1) Applying $\Hom_{\ce}(\cm,-)$ to $0\to X\to Y\to Z\to 0$, we have an exact sequence
\begin{eqnarray}
\Hom_{\ce}(\cm,Y)\xrightarrow{g}\Hom_{\ce}(\cm,Z)\to\Ext^1_{\ce}(\cm,X)\xrightarrow{f}\Ext^1_{\ce}(\cm,Y).
\end{eqnarray}
Thus we know that $\Hom_{\ce}(\cm,g)$ is surjective if and only if
$\Ext^1_{\ce}(\cm,f)$ is injective.
Using the Auslander-Reiten duality, this holds if and only if $\overline{\Hom}_{\ce}(f,\tau\cm)$ is surjective, which holds if and only if $\overline{\Hom}_{\ce}(f,\cm)$ is surjective. This holds if and only if $\Hom_{\ce}(f,\cm)$ is surjective.

(2) One can easily check (e.g. by using \cite[1.4,1.7]{DRSS}) that exact
sequences fulfilling the equivalent conditions in (1)
satisfy the axioms of exact categories in which any
object of $\add \cm$ is a projective and an injective object
(see \cite[2.28]{Kalck} for details).

We will show that $\ce$ has enough projectives with respect to this exact structure. For any $X\in\ce$, we take a right $\cm$-approximation $f\colon N'\to X$ of $X$. Since $\cm$ contains $\cp$, any morphism from $\cp$ to $X$ factors through $f$.
By a version of \ref{technical lemma} for exact categories, we have an exact sequence
\[
0\to Y\to N'\oplus P\xrightarrow{{f\choose 0}} X\to0.
\]
in $\ce$ with $P\in\cp$. This sequence shows that $\ce$ has enough projectives with respect to this exact structure.

Dually we have that $\ce$ has enough injectives.
Moreover, both projective objects and injective objects are $\add\cm$.
Thus the assertion holds.
\end{proof}

\section{Frobenius structures on special Cohen--Macaulay modules}\label{SCMsection}
Throughout this section we let $R$ denote a complete local rational surface singularity over an algebraically closed field of characteristic zero. Because of the characteristic zero assumption, rational singularities are always Cohen--Macaulay.  We refer the reader to \S\ref{uncount section}  for more details regarding  rational surface singularities.

We denote $\CM(R)$ to be the category of maximal Cohen--Macaulay (=CM) $R$-modules.  Since $R$ is normal and two-dimensional, a module is CM if and only if it is reflexive.  The category $\CM(R)$, and all subcategories thereof, are Krull--Schmidt categories since $R$ is complete local. One such subcategory is the category of \emph{special} CM modules, denoted $\SCM(R)$, which consists of all those CM $R$-modules $X$ satisfying $\Ext^1_R(X,R)=0$.

The category $\SCM(R)$ is intimately related to the geometry of $\Spec R$.  If we denote the minimal resolution of $\Spec R$ by
\[
Y\stackrel{\pi}{\longrightarrow}\Spec R,
\]
and define $\{E_i\}_{i\in I}$ to be the set of exceptional curves, then the following is well known:

\begin{prop}\label{known1}\t{(1)} There are only finitely many indecomposable objects in $\SCM(R)$.\\
\t{(2)} Indecomposable non-free objects in $\SCM(R)$ correspond bijectively to $\{E_i\}_{i\in I}$.
\end{prop}
\begin{proof}
(2) is Wunram \cite[1.2]{Wun} (using \cite[2.7]{IW1} to show that definition of special in \cite{Wun} is the same as the one used here), and (1) is a consequence of (2).
\end{proof}

Thus by \ref{known1}(2) $\SCM(R)$ has an additive generator $M:=R\oplus\bigoplus_{i\in I}M_i$, where as in the introduction by convention $M_i$ is the indecomposable special CM module corresponding to $E_i$.  The corresponding endomorphism ring $\Lambda:=\End_R(M)$ is called the \emph{reconstruction algebra} of $R$, see \cite{ReconA,IW1}.  The following is also well-known.

\begin{prop}\label{glrecon3}
Consider the reconstruction algebra $\Lambda$.  Then 
\[
\gl\Lambda=\left\{ \begin{array}{cl} 2 & \mbox{if $R$ is Gorenstein}\\ 3 & \mbox{else.}\end{array}  \right. 
\]
\end{prop}
\begin{proof}
An algebraic proof can be found in \cite[2.10]{IW1} or \cite[2.6]{IW2}.  A geometric proof can be found in \cite{GL2}.
\end{proof}

\begin{remark}\label{CM not CM}
The reconstruction algebra $\Lambda$, and some of the $e\Lambda e$ below, will turn out to be Iwanaga--Gorenstein in \S\ref{geometry}.  However we remark here that $\Lambda$ is usually \emph{not} Gorenstein in the stronger sense that $\omega_\Lambda:=\Hom_R(\Lambda,\omega_R)$ is a projective $\Lambda$-module.  Thus, unfortunately the objects of $\GP(\Lambda)$ are not simply those $\Lambda$-modules that are CM as $R$-modules, i.e.\  $\GP(\Lambda)\subsetneq\{X\in\mod(\Lambda)\ |\ X\in\CM(R)\}$ in general.  In this paper we will always reserve `CM' to  mean CM as an $R$-module, and this is why we use the terminology `GP' (=Gorenstein projective) for non-commutative Iwanaga--Gorenstein rings.
\end{remark}

We will be considering many different factor categories of $\SCM(R)$, so as to avoid confusion we now fix some notation.

\begin{defin}
Let $X\in\SCM(R)$.  We define the factor category $\underline{\SCM}_X(R)$ to be the category consisting of the same objects as $\SCM(R)$, but where
\[
\Hom_{\underline{\SCM}_X(R)}(a,b):=\frac{\Hom_{\SCM(R)}(a,b)}{\mathcal{X}(a,b)},
\]
where $\mathcal{X}(a,b)$ is the subgroup of morphisms $a\to b$ which factor through an element in $\add X$.
\end{defin}

As in the introduction, we consider the module $D:=R\oplus\left(\bigoplus_{d\in\mathcal{D}}M_d\right)$.  Algebraically the following is known; the geometric properties of $D$ will be established in \ref{D is derived equiv} and \ref{total discrepant} later.

\begin{prop}\label{known2}
\t{(1)} The category $\SCM(R)$ has the natural structure of a Frobenius category, whose projective objects are precisely the objects of $\add D$.  Consequently $\underline{\SCM}_D(R)$ is a triangulated category.\\
\t{(2)} For any indecomposable object $X$ in $\underline{\SCM}_D(R)$, there exists an AR triangle of the form
\[
X\to E\to X\to X[1].
\]
\t{(3)} The stable category $\underline{\SCM}_D(R)$ has a Serre functor $\mathbb{S}$ such that $\mathbb{S}X\simeq X[1]$ for any $X\in\underline{\SCM}_D(R)$.
\end{prop}
\begin{proof}
(1) The exact sequences are defined using the embedding $\SCM(R)\subseteq \mod R$, and the result follows from \cite[4.2]{IW2}.\\
(2) is \cite[4.9]{IW2}.\\
(3) $\underline{\SCM}_D(R)$ has AR triangles by (2), so there exists a Serre functor by \cite[I.2.3]{RVdB} such that
\[
\tau X\to E\to X\to\mathbb{S}X
\]
is the AR triangle.  By inspection of (2), we see that $\mathbb{S}X[-1]\simeq X$.
\end{proof}
\begin{remark}
The above proposition more-or-less asserts that the category $\underline{\SCM}_D(R)$ is 1-Calabi--Yau, but it does not show that the isomorphism in \ref{known2}(3) is functorial.  We prove that it is functorial in \ref{main triangles}, using geometric arguments.
\end{remark}

The following important observation, which generalises \ref{known2}(1), is obtained by applying~\ref{new Frobenius structure} to $(\ce,\cm,\tau)=(\SCM(R),\add N,\mathbb{S}[-1])$.

\begin{prop}\label{new Frobenius}
Let $N\in\SCM(R)$ such that $\add D\subseteq\add N$. Then\\
\t{(1)} Let $0\to X\xrightarrow{f} Y\xrightarrow{g} Z\to 0$ be an exact sequence of $R$-modules with $X,Y,Z\in\SCM(R)$. Then $\Hom_R(N,g)$ is surjective if and only if $\Hom_R(f,N)$ is surjective.\\
\t{(2)} $\SCM(R)$ has the structure of a Frobenius category whose projective objects are exactly $\add N$.  We denote it by $\SCM_N(R)$. More precisely, the short exact sequences of $\SCM_N(R)$ are the short exact sequences $0\to X\xrightarrow{f} Y\xrightarrow{g} Z\to 0$ of $R$-modules such that $\Hom_R(f,N)$ is surjective.
\end{prop}

We maintain the notation from above, in particular $\Lambda:=\End_R(M)$ is the reconstruction algebra, where $M:=R\oplus\left(\bigoplus_{i\in I}M_i\right)$, and $D:=R\oplus\left(\bigoplus_{d\in\mathcal{D}}M_d\right)$.  For any  summand $N$ of $M$, we denote $e_N$ to be the idempotent in $\Lambda$ corresponding to the  summand $N$.  The following is the main result of this section.

\begin{thm}\label{main}
Let $N\in\SCM(R)$ such that $\add D\subseteq\add N$. Then\\
\t{(1)} $e_N\Lambda e_N=\End_R(N)$ is an Iwanaga--Gorenstein ring of dimension at most three.\\
\t{(2)} There is an equivalence $\Hom_R(N,-)\colon\SCM(R)\to\GP(\End_R(N))$ that induces a triangle equivalence
\[
\underline{\SCM}_N(R)\simeq\underline{\GP}(\End_R(N)).
\]
\end{thm}
\begin{proof}
By \ref{new Frobenius}(2), $\SCM(R)$ has the structure of a Frobenius category in which $\proj\SCM(R)=\add N$.  Since $\SCM(R)$ has finite type, there is some $X\in\SCM(R)$ such that $\add(N\oplus X)=\SCM(R)$, in which case $\End_R(N\oplus X)$ is Morita equivalent to the reconstruction algebra, so $\gl\End_R(N\oplus X)\leq 3$ by \ref{glrecon3}.  Hence (1) follows from \ref{t:main-thm}(1), and since $\SCM(R)$ is idempotent complete, (2) follows from \ref{t:main-thm}(2).
\end{proof}

\begin{remark}
We will give an entirely geometric proof of \ref{main}(1) in \S\ref{geometry}, which also holds in greater generality.
\end{remark}

The following corollary will be strengthened in \ref{stengthenedGLdim} later.

\begin{cor}\label{cor to strengthen}
Let $N\in\SCM(R)$ such that $\add D\subseteq\add N\subsetneq\add M$. Then $e_N\Lambda e_N=\End_R(N)$ has infinite global dimension.
\end{cor}
\begin{proof}
By \ref{main}(2), we know that $\underline{\SCM}_N(R)\simeq\underline{\GP}(\End_R(N))\simeq\Dsg(\End_R(N))$, where the last equivalence holds by Buchweitz \cite[4.4.1(2)]{Buch}.  It is clear that $\underline{\SCM}_N(R)\neq 0$ since $\add N\subsetneq \add M$.  Hence $\Dsg(\End_R(N))\neq 0$, which is well-known to imply that $\gl\End_R(N)=\infty$.
\end{proof}

\section{Relationship to partial resolutions of rational surface singularities}\label{geometry}

We show in \S\ref{scheme give alg} that if an algebra $\Gamma$ is derived equivalent to a Gorenstein scheme that is projective birational over a CM ring, then $\Gamma$ is Iwanaga--Gorenstein.  In \S\ref{tilt on partial} we then exhibit algebras derived equivalent to partial resolutions of rational surface singularities, and we then use this information to strengthen many of our previous results. 

In this section we will assume that all schemes $Y$ are noetherian, separated, normal CM, of pure Krull dimension $d<\infty$, and are finite type over a field $k$.  This implies that $\D(\Qcoh Y)$ is compactly generated, with compact objects precisely the perfect complexes $\Perf(Y)$ \cite[2.5, 2.3]{Neeman}, and $\omega_Y=\gsh k[-\dim Y]$ where $g\colon Y\to\Spec k$ is the structure morphism.

\subsection{Gorenstein schemes and Iwanaga--Gorenstein rings}\label{scheme give alg}
Serre functors are somewhat more subtle in the singular setting.  Recall from \cite[7.2.6]{Ginz} the following.

\begin{defin}
Suppose that $Y\to\Spec S$ is a projective birational map where $S$ is a CM ring with canonical module $\omega_S$.  We say that a functor $\mathbb{S}\colon\Perf(Y)\to\Perf(Y)$ is a \emph{Serre functor relative to $\omega_S$} if there are functorial isomorphisms
\[
\RHom_S(\RHom_Y(\cf,\cg),\omega_S)\cong \RHom_Y(\cg,\mathbb{S}(\cf))
\]
in $\D(\Mod S)$ for all $\cf,\cg\in\Perf(Y)$.
\end{defin}

\begin{lemma}\label{Serre for alg}
Let $\Gamma$ be a module finite $S$-algebra, where $S$ is a CM ring with canonical module $\omega_S$, and suppose that there exists a functor $\mathbb{T}\colon\Kb(\proj\Gamma)\to\Kb(\proj\Gamma)$ such that 
\[
\RHom_S(\RHom_\Gamma(a,b),\omega_S)\cong \RHom_\Gamma(b,\mathbb{T}(a))
\]
for all $a,b\in\Kb(\proj\Gamma)$.  Then $\id\Gamma_\Gamma<\infty$. 
\end{lemma}
\begin{proof}
Denote $(-)^\dagger:=\RHom_S(-,\omega_S)$. We first claim that $\Gamma^\dagger\in\Kb(\Inj\Gamma^{\op})$.  By taking an injective resolution of $\omega_S$
\[
0\to\omega_S\to I_0\to \hdots\to I_d\to 0
\]
and applying $\Hom_S(\Gamma,-)$ we see that $\Gamma^\dagger$ is given as the complex
\[
\hdots\to 0\to\Hom_S(\Gamma, I_0)\to \hdots\to \Hom_S(\Gamma,I_d)\to 0\to\hdots.
\]
Since $\Hom_\Gamma(-,\Hom_S(\Gamma,I_i))=\Hom_S(\Gamma\otimes_\Gamma-,I_i)$ is an exact functor, each $\Hom_S(\Gamma,I_i)$ is an injective $\Gamma^{\op}$-module.  Hence $\Gamma^\dagger\in\Kb(\Inj\Gamma^{\op})$, as claimed.

Now $\mathbb{T}(\Gamma)\in\Kb(\proj\Gamma)$, and further
\[
\mathbb{T}(\Gamma)\cong\RHom_{\Gamma}(\Gamma,\mathbb{T}(\Gamma))\cong \RHom_\Gamma(\Gamma,\Gamma)^\dagger=\Gamma^\dagger.
\]
Hence $\Gamma^\dagger\in\Kb(\proj\Gamma)=\thick(\Gamma)$ and so $\Gamma=\Gamma^{\dagger\dagger}\in\thick(\Gamma^\dagger)\subseteq\Kb(\Inj\Gamma^{\op})$.  This shows that $\Gamma$ has finite injective dimension as a $\Gamma^{\op}$-module, i.e.\ as a right $\Gamma$-module.
\end{proof}

Grothendieck duality gives us the following.
\begin{thm}\label{Groth duality}
Let $Y\to\Spec S$ be a projective birational map where $S$ a CM ring with canonical module $\omega_S$.  Suppose that $Y$ is Gorenstein, then the functor $\mathbb{S}:=\omega_Y\otimes-\colon\Perf(Y)\to\Perf(Y)$ is a Serre functor relative to $\omega_S$.
\end{thm}
\begin{proof}
Since $Y$ is Gorenstein, the canonical sheaf $\omega_Y$ is locally free, and hence $\mathbb{S}:=\omega_Y\otimes-=\omega_Y\otimes^{\mathbf{L}}-$ does indeed take $\Perf(Y)$ to $\Perf(Y)$.  Also, $\omega_Y=\fsh\omega_S$ and so 
\begin{eqnarray*}
\RHom_Y(\cg,\mathbb{S}(\cf))=\RHom_Y(\cg,\cf\otimes^{\mathbf{L}}\omega_Y)&\cong&\RHom_Y(\RsHom_Y(\cf,\cg),\omega_Y)\\
&\cong&\RHom_Y(\RsHom_Y(\cf,\cg),\fsh\omega_S)\\
&\cong&\RHom_S(\Rf\RsHom_Y(\cf,\cg),\omega_S)\\
&\cong&\RHom_S(\RHom_Y(\cf,\cg),\omega_S)
\end{eqnarray*}
for all $\cf,\cg\in\Perf(Y)$, where the second-last isomorphism is Grothendieck duality.
\end{proof}
The last two results combine to give the following, which is the main result of this subsection.
\begin{cor}\label{Db to Gor is Gor}
Let $Y\to\Spec S$ be a projective birational map where $S$ is a CM ring with canonical module $\omega_S$.  Suppose that $Y$ is derived equivalent to $\Gamma$.  Then if $Y$ is a Gorenstein scheme, $\Gamma$ is an Iwanaga--Gorenstein ring.
\end{cor}
\begin{proof}
By \ref{Groth duality} there is a Serre functor $\mathbb{S}:\Perf(Y)\to\Perf(Y)$ relative to $\omega_S$.  By \cite[4.12]{IW5} this induces a Serre functor relative to $\omega_S$ on $\Kb(\proj \Gamma)$.  Hence \ref{Serre for alg} shows that $\id\Gamma_\Gamma<\infty$.  

Repeating the argument with $\cv^\vee:=\RsHom_Y(\cv,\co_Y)$, which is well-known to give an equivalence between $Y$ and $\Gamma^{\op}$ (see e.g.\ \cite[2.6]{BH}), we obtain an induced Serre functor relative to $\omega_S$ on $\Kb(\proj\Gamma^{\op})$. Applying \ref{Serre for alg} to $\Gamma^{\op}$ shows that $\id_\Gamma\Gamma<\infty$.  
\end{proof}

\subsection{Tilting bundles on partial resolutions}\label{tilt on partial}
We now return to the setup in \S\ref{SCMsection}, namely $R$ denotes a complete local rational surface singularity over an algebraically closed field of characteristic zero.  We inspect the exceptional divisors in $Y$, the minimal resolution of $\Spec R$. Recall from the introduction that we have $I=\cc\cup\cd$ where $\cc$ are the crepant curves and $\cd$ are the discrepant curves.  We choose a subset $\cs\subseteq I$, and contract all curves in $\cs$. In this way we obtain a scheme which we will denote $X^{\cs}$ (see for example \cite[\S 4.15]{Reid93}). In fact, the minimal resolution $\pi\colon Y\to\Spec R$ factors as
\[
Y\xrightarrow{f^{\cs}} X^{\cs}\xrightarrow{g^{\cs}}\Spec R.
\]
When $\cs\subseteq \cc$ then $f^{\cs}$ is crepant and further $X^{\cs}$ has only isolated ADE singularities since we have contracted only ($-2$)-curves --- it is well-known that in the dual graph of the minimal resolution, all maximal ($-2$)-curves must lie in ADE configurations (see e.g.\ \cite[3.2]{TT}).

\begin{example}\label{T9example}
To make this concrete, consider the $\mathbb{T}_9$ singularity \cite[p47]{Riemen} $\Spec R=\mathbb{C}^2/\mathbb{T}_9$, which has minimal resolution
\[
Y:=
\begin{array}{c}
\begin{tikzpicture} 
\draw (0,0,0) to [bend left=25] node[above right] {$\scriptstyle E_2$}  node[below right] {$\scriptstyle -3$} (2,0,0);
\draw (1.8,0,0) to [bend left=25] node[below] {$\scriptstyle -2$} node[above] {$\scriptstyle E_3$} (4,0,0);
\draw (3.8,0,0) to [bend left=25] node[below] {$\scriptstyle -2$} node[above] {$\scriptstyle E_4$} (6,0,0);
\draw (-1.8,0,0) to [bend left=25] node[above] {$\scriptstyle E_1$} node[below] {$\scriptstyle -3$} (0.2,0,0);
\draw (-0.2,1.5,0) to [bend left=25] node[left]{$\scriptstyle -2$} node[above right] {$\scriptstyle E_5$} (0.8,0,0);
\end{tikzpicture}
\end{array} 
\longrightarrow \Spec R
\]
so $\mathcal{C}=\{ E_3,E_4,E_5\}$.  Choosing $\mathcal{S}=\{ E_3,E_5\}$ gives
\[
X^{\mathcal{S}}:=
\begin{array}{c}
\begin{tikzpicture} 
\draw (0,0,0) to [bend left=25] node[above right] {$\scriptstyle E_2$} (2,0,0);
\draw (-1.8,0,0) to [bend left=25] node[above] {$\scriptstyle E_1$} (0.2,0,0);
\filldraw [black] (0.8,0.225,0) circle (1pt);
\filldraw [black] (1.9,0.045,0) circle (1pt);
\node at (0.8,0,0) {$\scriptstyle \frac{1}{2}(1,1)$};
\node at (2,-0.175,0) {$\scriptstyle \frac{1}{2}(1,1)$};
\draw (1.8,0,0) to [bend left=25] node[above] {$\scriptstyle E_4$} (4,0,0);
\end{tikzpicture} 
\end{array}
\]
where $\frac{1}{2}(1,1)$ is complete locally the $A_1$ surface singularity.  On the other hand, choosing $\mathcal{S}=\mathcal{C}=\{ E_3,E_4,E_5\}$ gives
\[
X^{\mathcal{C}}:=
\begin{array}{c}
\begin{tikzpicture} 
\draw (0,0,0) to [bend left=25] node[above right] {$\scriptstyle E_2$} (2,0,0);
\draw (-1.8,0,0) to [bend left=25] node[above] {$\scriptstyle E_1$} (0.2,0,0);
\filldraw [black] (0.8,0.225,0) circle (1pt);
\filldraw [black] (1.9,0.045,0) circle (1pt);
\node at (0.8,0,0) {$\scriptstyle \frac{1}{2}(1,1)$};
\node at (2,-0.175,0) {$\scriptstyle \frac{1}{3}(1,2)$};
\end{tikzpicture} 
\end{array}
\]
Note in particular that in these cases $\cs\subseteq\cc$ so $\Sing X^{\cs}$ always has only finitely many points, and each is Gorenstein ADE.
\end{example}

The following is well--known to experts and is somewhat implicit in the literature.  For lack of any reference, we provide a proof here.  As before, $\Lambda$ denotes the reconstruction algebra. 
\begin{thm}\label{main Db}
Let $\cs\subseteq I$, set $N^{\cs}:=R\oplus(\bigoplus_{i\in I\backslash \cs}M_i)$ and let $e$ be the idempotent in $\Lambda$ corresponding to $N^{\cs}$.  Then $e\Lambda e=\End_R(N^{\cs})$ is derived equivalent to $X^{\cs}$ via a tilting bundle $\cv_\cs$ in such a way that 
\[
\begin{array}{c}
{\SelectTips{cm}{10}
\xy0;/r.4pc/:
(-10,20)*+{\Db(\mod\Lambda)}="A2",(20,20)*+{\Db(\coh Y)}="A3",
(-10,10)*+{\Db(\mod e\Lambda e)}="a2",(20,10)*+{\Db(\coh X^{\cs})}="a3",
\ar"A3";"A2"_{\RHom_Y(\cv_\emptyset,-)}
\ar"A2";"a2"_{e(-)}
\ar"a3";"a2"_{\RHom_{X^\cs}(\cv_{\cs},-)}
\ar"A3";"a3"^{\Rfcs}
\endxy}
\end{array}
\]
commutes.
\end{thm}
\begin{proof}
Since all the fibres are at most one-dimensional and $R$ has rational singularities, by \cite[Thm.\ B]{VdB1d} there is a tilting bundle on $Y$ given as follows:  let $E=\pi^{-1}(\m)$ where $\m$ is the unique closed point of $\Spec R$.  Giving $E$ the reduced scheme structure, write $E_{\rm red}=\cup_{i\in I}E_i$, and let $\cl^Y_i$ denote the line bundle on $Y$ such that $\cl^Y_i\cdot E_j=\delta_{ij}$.  If the multiplicity of $E_i$ in $E$ is equal to one, set $\cm^Y_i:=\cl^Y_i$ \cite[3.5.4]{VdB1d}, else define $\cm^Y_i$ to be given by the maximal extension
\[
0\to\co_Y^{\oplus(r_i-1)}\to\cm^Y_i\to\cl^Y_i\to 0
\]
associated to a minimal set of $r_i-1$ generators of $H^1(Y,(\cl^Y_i)^{-1})$.  Then $\cv_{\emptyset}:=\co_Y\oplus(\bigoplus_{i\in I}\cm^Y_i)$ is a tilting bundle on $Y$ \cite[3.5.5]{VdB1d}.
 
To ease notation denote $X:={X^{\cs}}$, and further denote $Y\xrightarrow{f^{\cs}} X^{\cs}\xrightarrow{g^{\cs}}\Spec R$ by 
\[
Y\xrightarrow{f} X\xrightarrow{g}\Spec R.
\]
Then in an identical manner to the above, $\cv_\cs:=\co_{X}\oplus(\bigoplus_{i\in I\backslash \cs}\cm^{X}_i)$ is a tilting bundle on $X$.

We claim that $f^*(\cv_\cs)=\co_Y\oplus(\bigoplus_{i\in I\backslash \cs}\cm^Y_i)$.  Certainly $f^*\cl^X_i=\cl^Y_i$ for all $i\in I\backslash \cs$, and pulling back
\[
0\to\co_X^{\oplus(r_i-1)}\to\cm^X_i\to\cl^X_i\to 0
\]
gives an exact sequence
\begin{eqnarray}
0\to\co_Y^{\oplus(r_i-1)}\to f^*\cm^X_i\to\cl^Y_i\to 0.\label{On top}
\end{eqnarray}
But 
\[
\Ext^1_Y(f^*\cm^X_i,\co_Y)=\Ext^1_Y(\Lf\cm^X_i,\co_Y)=\Ext^1_X(\cm^X_i,\Rf\co_Y)=\Ext^1_X(\cm^X_i,\co_X),
\]
which equals zero since $\cv_\cs$ is tilting.  Hence (\ref{On top}) is a maximal extension, so it follows (by construction) that $\cm^Y_i\cong f^*\cm^X_i$ for all $i\in I\backslash \cs$, so $f^*(\cv_\cs)=\co_Y\oplus_{i\in I\backslash \cs}\cm^Y_i$ as claimed.

Now by the projection formula 
\[
\Rf(f^*\cv_\cs)\cong \Rf(\co_Y\otimes f^*\cv_\cs)\cong \Rf(\co_Y)\otimes \cv_\cs\cong\co_X\otimes\cv_\cs=\cv_\cs,
\]
and so it follows that
\[
\End_X(\cv_\cs)\cong\Hom_X(\cv_\cs,\Rf(f^*\cv_\cs))\cong\Hom_Y(\Lf\cv_\cs,f^*\cv_\cs)\cong\End_Y(f^*\cv_\cs),
\]
i.e.\ $\End_X(\cv_\cs)\cong\End_Y(\co_Y\oplus_{i\in I\backslash \cs}\cm^Y_i)$.  But it is very well-known (see e.g.\ \cite[3.2]{GL2}) that $\End_Y(\co_Y\oplus_{i\in I\backslash \cs}\cm^Y_i)\cong \End_R(R\oplus_{i\in I\backslash\cs} M_i)=\End_R(N^{\cs})$.

Hence we have shown that $\cv_\cs$ is a tilting bundle on $X^\cs$ with endomorphism ring isomorphic to $\End_R(N^{\cs})$, so the first statement follows.  For the last statement, simply observe that we have functorial isomorphisms
\begin{eqnarray*}
\RHom_{X^\cs}(\cv_\cs,\Rf(-))&=&\RHom_{Y}(\Lf\cv_\cs,-)\\
&=& \RHom_{Y}(\co_Y\oplus_{i\in I\backslash \cs}\cm^Y_i,-)\\
&=& e\RHom_{Y}(\co_Y\oplus_{i\in I}\cm^Y_i,-)\\
&=& e\RHom_{Y}(\cv_\emptyset,-).
\end{eqnarray*}
\end{proof}
\begin{remark}\label{stengthenedGLdim}
The above \ref{main Db} shows that if $\Lambda$ is the reconstruction algebra and $e\neq 1$ is a non-zero idempotent containing the idempotent corresponding to $R$, then $e\Lambda e$ always has infinite global dimension, since it is derived equivalent to a singular variety.   This greatly generalizes \ref{cor to strengthen}, which only deals with idempotents corresponding to partial resolutions `above' $X^{\mathcal{C}}$; these generically do not exist.  It would be useful to have a purely algebraic proof of the fact $\gl e\Lambda e=\infty$, since this is related to many problems in higher dimensions.
\end{remark}

Now recall from \ref{module N} that $D:=R\oplus(\bigoplus_{d\in\mathcal{D}}M_d)$.  This is just $N^{\cc}$, so as the special case of \ref{main Db} when $\cs=\cc$ we obtain the following.

\begin{cor}\label{D is derived equiv}
$\End_R(D)$ is derived equivalent to $X^{\cc}$.
\end{cor}
\begin{remark}\label{total discrepant}
It follows from \ref{D is derived equiv} that the module $D$ corresponds to the largest totally discrepant partial resolution of $\Spec R$, in that any further resolution must involve crepant curves.  This scheme was much studied in earlier works (e.g.\ \cite{RRW}), and is related to the deformation theory of $\Spec R$.  We remark that $X^\cc$ is often referred to as the rational double point resolution. 
\end{remark}

As a further consequence of \ref{main Db}, we have the following.

\begin{thm}\label{main triangles}
If $\cs\subseteq \cc$, then we have triangle equivalences 
\[
\underline{\SCM}_{N^{\cs}}(R)\simeq\underline{\GP}(\End_R(N^{\cs}))\simeq\Dsg(\End_R(N^{\cs}))\simeq\Dsg(X^{\cs})\simeq\bigoplus_{x\in\Sing X^{\cs}}\uCM(\widehat{\mathcal{O}}_{X^{\cs}, x}),
\]
where $\Sing X^{\mathcal{S}}$ denotes the set of singular points of $X^{\mathcal{S}}$.  In particular, $\underline{\SCM}_{N^{\mathcal{S}}}(R)$ is 1-Calabi--Yau, and its shift functor satisfies $[2]=\mathrm{id}$.
\end{thm}
\begin{proof}
Since $R$ is complete local we already know that $\underline{\SCM}_{N^{\mathcal{S}}}(R)$ is idempotent complete, so the first equivalence is just \ref{main}(2).  Since $\End_R(N^{\mathcal{S}})$ is Iwanaga--Gorenstein by \ref{main}(1), the second equivalence is a well--known theorem of Buchweitz \cite[4.4.1(2)]{Buch}.  The third equivalence follows immediately from \ref{main Db} (see e.g.\ \cite[4.1]{IW5}).  The fourth equivalence follows from \cite{OrlovCompletion}, \cite{BK} or \cite[3.2]{IW5} since the singularities of $X^\cs$ are isolated and the completeness of $R$ implies that $\Dsg(X^{\cs})\simeq\underline{\SCM}_{N^{\cs}}(R)$ is idempotent complete. The final two statements hold since each $\widehat{\mathcal{O}}_{X^{\mathcal{S}}, x}$ is Gorenstein ADE, and for these it is well-known that $\uCM(\widehat{\mathcal{O}}_{X^{\mathcal{S}}, x})$ are 1-Calabi--Yau \cite{Auslander78}, satisfying $[2]=\mathrm{id}$ \cite{Eisenbud}.
\end{proof}

\begin{example}
In the previous example (\ref{T9example}) choose $\mathcal{S}=\{ E_3,E_5\}$, then by \ref{main triangles}
\[
\underline{\SCM}_{N^{\cs}}(R)\simeq \uCM\mathbb{C}[[x,y]]^{\frac{1}{2}(1,1)}\oplus \uCM\mathbb{C}[[x,y]]^{\frac{1}{2}(1,1)}.
\]
\end{example}

\begin{remark} It was remarked in \cite[4.14]{IW2} that often the category $\underline{\SCM}_D(R)$ is equivalent to that of a Gorenstein ADE singularity, but this equivalence was only known to be an additive equivalence, as the triangle structure on $\underline{\SCM}_D(R)$ was difficult to control algebraically.  The above \ref{main triangles} improves this by lifting the additive equivalence to a triangle equivalence. It furthermore generalises the equivalence to other Frobenius quotients of $\SCM(R)$ that were not considered in \cite{IW2}.
\end{remark}

We now use \ref{main triangles} to extend Auslander's algebraic  McKay correspondence.  This requires the notion of the dual graph relative to a morphism.

\begin{defin}\label{dualG}
Consider $f^{\cs}\colon Y\to X^{\cs}$.  The \emph{dual graph} with respect to $f^{\cs}$ is defined as follows: for each irreducible curve contracted by $f^{\cs}$ draw a vertex, and join two vertices if and only if the corresponding curves in $Y$ intersect.  Furthermore, label every vertex with the self-intersection number of the corresponding curve.
\end{defin}

The following, which is immediate from \ref{main triangles}, extends \cite[4.11]{IW2}.

\begin{cor}\label{McKay rel}
If $\cs\subseteq\cc$, then the AR quiver of the category $\underline{\SCM}_{N^{\cs}}(R)$ is the double of the dual graph with respect to the morphism $Y\to X^\cs$.
\end{cor}

\subsection{Iwanaga--Gorenstein rings from surfaces}\label{4.3}

The following corollary of \ref{main Db} gives a geometric proof of \ref{main}(1).

\begin{cor}\label{Goren from geom}
Let $N\in\SCM(R)$ such that $\add D\subseteq\add N$. Then $e_N\Lambda e_N=\End_R(N)$ is an Iwanaga--Gorenstein ring.
\end{cor}
\begin{proof}
Since $\add D\subseteq\add N$, \ref{main Db} shows that the algebra $\End_R(N)$ is derived equivalent, via a tilting bundle, to the Gorenstein scheme $X^{\cs}$.  Thus the result follows by \ref{Db to Gor is Gor}.
\end{proof}

The point is that using the geometry we can sharpen \ref{main}(1) and \ref{Goren from geom}, since we are explicitly able to determine the value of the injective dimension.  {The proof requires the following two lemmas, which we state and prove in greater generality.

\begin{lemma}\label{Nakayama}
Suppose that $(S,\m)$ is local, $\Gamma$ is a module--finite $S$-algebra, and $X,Y\in\mod \Gamma$.  Then $\Ext^{i}_\Gamma(X,Y)=0$ if $i>\id_\Gamma Y-\depth_S X$.
\end{lemma}
\begin{proof}
Use induction on $t=\depth_S X$. The case $t=0$ is clear.
Take an $X$-regular element $r$ and consider the sequence
\[
0 \to X \stackrel{r}{\to} X \to X/rX \to 0
\]
By induction we have $\Ext^{i+1}_\Gamma(X/rX,Y)=0$ for $i>\id_\Gamma Y-t$.
By the exact sequence
\[
\Ext^{i}_\Gamma(X,Y) \stackrel{r}{\to} \Ext^{i}_\Gamma(X,Y) \to \Ext^{i+1}_\Gamma(X/rX,Y)=0
\]
and Nakayama's Lemma, we have $\Ext^{i}_\Gamma(X,Y)=0$.
\end{proof}

Recall that if $\Gamma$ is an $S$-order, then we denote $\CM(\Gamma)$ to be the category consisting of those $X\in\mod\Gamma$ for which $X\in\CM(S)$. 

\begin{lemma}\label{AB type for id}\cite[Proposition 1.1(3)]{GN01} Suppose that $S$ is an equicodimensional (i.e.\ $\dim S=\dim S_{\mathfrak{m}}$ for all $\mathfrak{m}\in\Max S$) $d$-dimensional CM ring with canonical module $\omega_S$, and let $\Gamma$ be an $S$-order.  Then\\
\t{(1)} $\id_\Gamma \Hom_S(\Gamma,\omega_S)=d=\id_{\Gamma^{\op} }\Hom_S(\Gamma,\omega_S)$.\\
\t{(2)} $\id_\Gamma X=\pd_{\Gamma^{\op}}\Hom_S(X,\omega_S)+d$ for all $X\in\CM(\Gamma)$.
\end{lemma}
\begin{proof}
We include a proof for the convenience of the reader.
To simplify notation denote $\Hom_S(-,\omega_S):=(-)^\dagger$. This gives an exact duality $\CM(\Gamma)\leftrightarrow\CM(\Gamma^{\op})$. The statements are local, so we can assume that $S$ is a local ring.\\
\t{(1)}  Consider the minimal injective resolution of $\omega_S$ in $\mod S$, namely
\[
0\to\omega_S\to I_0\to I_1\to\hdots\to I_d\to 0.
\]
Applying $\Hom_S(\Gamma,-)$, using the fact that $\Gamma\in\CM(S)$ we obtain an exact sequence
\[
0\to\Gamma^\dagger\to \Hom_S(\Gamma,I_0)\to \hdots\to \Hom_S(\Gamma,I_d)\to 0.
\]
As in the proof of \ref{Serre for alg}, each $\Hom_S(\Gamma,I_i)$ is an injective $\Gamma$-module.  This shows that $\id_\Gamma\Gamma^\dagger\leq \dim S$.  If $\id_\Gamma\Gamma^\dagger<\dim S$ then
\begin{align}
0\to \Hom_S(\Gamma,\Omega^{-d+1}\omega_S)\to \Hom_S(\Gamma,I_{d-1})\to \Hom_S(\Gamma,I_d)\to 0\label{split ses}
\end{align}
must split.  Let $T$ be some non-zero $\Gamma$-module which has finite length as an $S$-module (e.g.\ $T=\Gamma/\m\Gamma$ for some $\m\in\Max S$).  Since (\ref{split ses}) splits, applying $\Hom_\Gamma(T,-)$ shows that the top row in the following commutative diagram is exact
\[
{\SelectTips{cm}{10}
\xy0;/r.37pc/:
(0,0)*+{0}="1",
(16,0)*+{\Hom_{\Gamma}(T,{}_S(\Gamma,\Omega^{-d+1}\omega_S))}="2",
(42,0)*+{\Hom_{\Gamma}(T,{}_S(\Gamma,I_{d-1}))}="3",
(66,0)*+{\Hom_{\Gamma}(T,{}_S(\Gamma,I_{d-1}))}="4",
(80,0)*+{0}="5",
(0.5,-7)*+{0}="b0",
(16,-7)*+{\Hom_S(T,\Omega^{-d+1}\omega_S)}="b1",
(42,-7)*+{\Hom_S(T,I_{d-1})}="b2",
(66,-7)*+{\Hom_S(T,I_{d})}="b3",
(80,-7)*+{0}="b4"
\ar"1";"2"
\ar"2";"3"
\ar"3";"4"
\ar"4";"5"
\ar"b0";"b1"
\ar"b1";"b2"
\ar"b2";"b3"
\ar"b3";"b4"
\ar_{\cong}"2";"b1"
\ar_{\cong}"3";"b2"
\ar_{\cong}"4";"b3"
\endxy}
\]
Hence the bottom row is exact.  But $T$ has finite length, so $\Hom_S(T,I_{d-1})=0$ since none of the associated primes of $I_{d-1}$ is maximal by equicodimensionality of $S$.  But by the above diagram this implies that $\Hom_S(T,I_d)=0$, which is a contradiction since $\Hom_S(-,I_d)$ is a duality on finite length modules.\\
(2)  Set $l:=\pd_{\Gamma^{\op}} X^\dagger$ and $m:=\id_\Gamma X$.    Consider a projective resolution of $X^\dagger$ over $\Gamma^{\op}$
\begin{eqnarray}
\hdots \xrightarrow{{ f_2}} P_{1}\xrightarrow{f_1} P_0\to X^\dagger\to 0,\label{min proj}
\end{eqnarray}
then applying $(-)^\dagger$ gives rise to an exact sequence
\begin{equation}
0\to X\to P_0^\dagger\xrightarrow{ {f_1^\dagger}} P_1^\dagger\xrightarrow{f_2^\dagger}\hdots.\label{Injy res}
\end{equation}
Since  by (1) each $P_i^\dagger$ has injective dimension $d$, it follows that $m=\id_\Gamma X\leq l+d$.  So $m$ is infinity implies that $l$ is infinity, and in this case the equality holds. Hence we can assume that $m<\infty$.  

We first claim that $m\geq d$.  This is true if $X\in\add\Gamma^\dagger$ by (1).  Now we assume that $X\notin\add\Gamma^\dagger$, so $X^\dagger\notin\add\Gamma$.  Thus
\[
0\neq\Ext^1_{\Gamma^{\op}}(X^\dagger,\Omega_{\Gamma^{\op}}X^\dagger)=\Ext^1_{\Gamma}((\Omega_{\Gamma^{\op}}X^\dagger)^\dagger,X).
\]
Since $\depth_S(\Omega_{\Gamma^{\op}}X^\dagger)^\dagger=d$, by \ref{Nakayama} we conclude that $m\geq d+1$. Thus we have $m\ge d$ in both cases.

Consider $\mathrm{Im}(f_{m-d+1}^{\dagger})$, then since $\depth_S(\mathrm{Im}(f_{m-d+1}^{\dagger}))=d$, by \ref{Nakayama} it follows that $\Ext^{m-d+1}_\Gamma(\mathrm{Im}(f_{m-d+1}^{\dagger}), X)=0$.  But since $X\in\CM(\Gamma)$ and the $P_i^\dagger$ are injective in $\CM(\Gamma)$, $\Ext^{j}_{\Gamma}(X,P_i^\dagger)=0$ for all $j>0$ and so (\ref{Injy res}) shows that
\[\Ext_\Gamma^1(\mathrm{Im}(f_{m-d+1}^{\dagger}),\mathrm{Im}(f_{m-d}^{\dagger}))=\hdots=\Ext_\Gamma^{m-d+1}(\mathrm{Im}(f_{m-d+1}^{\dagger}),X)=0.\]
This implies that the short exact sequence 
\[
0\to \mathrm{Im}(f_{m-d}^{\dagger})\to P_{m-d}^\dagger\to\mathrm{Im}(f_{m-d+1}^\dagger)\to 0
\]
splits, which in turn implies that the sequence
\[
0\to \mathrm{Im}(f_{m-d+1})\to P_{m-d}\to\mathrm{Im}(f_{m-d})\to 0
\]
splits, so $l\leq m-d$. In particular $l<\infty$, so we may assume that $P_i=0$ for $i>l$ in (\ref{min proj}).  So (\ref{Injy res}) shows that $m\leq l+d$.  Combining inequalities, we have $m=l+d$, as required.
\end{proof}

The following result is the main result in this subsection.  We remark that this gives a generalization of \ref{glrecon3}.

\begin{thm}\label{precise id}
Let $N\in\SCM(R)$ such that $\add D\subseteq\add N$ and put $\Gamma:=\End_R(N)$. Then 
\[
\id_\Gamma \Gamma=\left\{ \begin{array}{cl} 2 & \mbox{if $R$ is Gorenstein}\\ 3 & \mbox{else.}\end{array}  \right. 
\]
\end{thm}
\begin{proof} 
By \ref{AB type for id} we know that $\id_\Gamma\Gamma\geq 2$.\\
(1) Suppose that $R$ is Gorenstein.  In this case $\Gamma\in\CM(R)$ is a symmetric $R$-order, meaning $\Gamma\cong\Hom_R(\Gamma,R)$ as $\Gamma$-$\Gamma$ bimodules \cite[2.4(3)]{IR}.   Thus $\id_\Gamma\Gamma=\dim R=2$ by \ref{AB type for id}.\\
(2) Suppose that $R$ is not Gorenstein, so there exists an indecomposable  summand $N_i$ of $N$ such that $N_i$ corresponds to a non-($-2$)-curve.  Necessarily $N_i$ is not free, and further by \ref{known2}(1) $\Ext^1_R(N_i,X)=0$ for all $X\in\SCM(R)$.

Now if $\id_\Gamma\Gamma=\dim R=2$ then by \ref{AB type for id} $\Hom_R(\Gamma,\omega_R)$ is a projective $\Gamma$-module.  But
\[
\Hom_R(\Gamma,\omega_R)=\Hom_R(\End_R(N),\omega_R)\cong \Hom_R(N,(N\otimes_R\omega_R)^{**})\cong\Hom_R(N,\tau N)
\]
where $\tau$ is the AR translation in the category $\CM(R)$, and the middle isomorphism holds e.g.\ by \cite[4.1]{AG}. Hence by reflexive equivalence $\Hom_R(N,-)\colon\CM(R)\to\CM(\Gamma)$, we have $\tau N\in\add N$, so in particular $\tau N_i\in\SCM(R)$.  But this implies that $\Ext^1_R(N_i,\tau N_i)=0$ by above, which by the existence of AR sequences is impossible.   Hence $\id_\Gamma\Gamma\neq 2$.  Now \ref{main}(1) implies that $\id_\Gamma\Gamma\leq 3$ and so consequently $\id_\Gamma\Gamma=3$.
\end{proof}

\subsection{Construction of Iwanaga--Gorenstein rings}\label{uncount section} 

In this subsection, we work over $\mathbb{C}$. If $R$ is not Gorenstein and $N\in\SCM(R)$ such that $\add D\subseteq\add N$, then by \ref{main triangles} and \ref{precise id} $\Gamma:=\End_R(N)$ is an Iwanaga--Gorenstein ring with $\id\Gamma=3$, such that $\underline{\GP}(\Gamma)$ is a direct sum of stable CM categories of ADE singularities.  In particular, each $\Gamma$ has finite Gorenstein--projective type.  The simplest case is when $\Gamma$ has only one non-free indecomposable GP-module, i.e.\ the case $\underline{\GP}(\Gamma)\simeq\uCM (\mathbb{C}[[x,y]]^{\frac{1}{2}(1,1)})$. 

The purpose of this section is to prove the following theorem.

\begin{thm}\label{uncount G}
Let $G\leq\SL(2,\mathbb{C})$ be a finite subgroup, with $G\ncong E_8$.  Then there are uncountably many non-isomorphic Iwanaga--Gorenstein rings $\Gamma$ with $\id\Gamma=3$, such that $\underline{\GP}(\Gamma)\simeq\uCM(\mathbb{C}[[x,y]]^{G})$.
\end{thm}

The theorem is unusual, since commutative algebra constructions such as Kn\"orrer periodicity only give countably many non-isomorphic Gorenstein rings $S$ with $\uCM(S)\simeq\uCM(\mathbb{C}[[x,y]]^{G})$, and further no two of the $S$ have the same injective dimension.  

\begin{remark}
We remark that the omission of type $G\cong E_8$ from our theorem is also unusual; it may still be possible that there are uncountably many non-isomorphic Iwanaga--Gorenstein rings $\Gamma$ with $\id\Gamma=3$ such that $\underline{\GP}(\Gamma)\simeq\uCM(\mathbb[[x,y]]^{E_8})$, however our methods do not produce any. It is unclear to us whether this illustrates simply the limits of our techniques, or whether the finite type $E_8$ is much more rare.
\end{remark}

To prove \ref{uncount G} requires some knowledge of complete local rational surface singularities over $\mathbb{C}$, which we now review.  If $R$ is a complete local rational surface singularity, then if we consider the minimal resolution $Y\to\Spec R$, then (as before) the fibre above the origin is well-known to be a tree (i.e.\ a finite connected graph with no cycles) of $\mathbb{P}^1$s denoted $\{ E_i\}_{i\in I}$.  Their self-intersection numbers satisfy $E_i\cdot E_i\leq -2$, and moreover the intersection matrix $(E_i\cdot E_j)_{i,j\in I}$ is negative definite. We encode the intersection matrix in the form of the labelled dual graph:

\begin{defin}
We refer to the dual graph with respect to the morphism $Y\to\Spec R$ (in the sense of \ref{dualG}) as the \emph{dual graph} of $R$.
\end{defin}

Thus, given a complete local rational surface singularity, we obtain a labelled tree.  Before we state as a theorem the solution to the converse problem, we first require some notation.  

Suppose that $T$ is a tree, with vertices denoted $E_1,\hdots,E_n$, labelled with integers $w_1,\hdots,w_n$.  To this data we associate the symmetric matrix $M_T=(b_{ij})_{1\leq i,j\leq n}$ with $b_{ii}$ defined by $b_{ii}:=w_i$, and $b_{ij}$ (with $i\neq j$) defined to be the number of edges linking the vertices $E_i$ and $E_j$.  We denote the free abelian group generated by the vertices $E_i$ by $\cz$, and call its elements \emph{cycles}.  The matrix $M_T$ defines a symmetric bilinear form $(-,-)$ on $\cz$ and in analogy with the geometry, we will often write $Y\cdot Z$ instead of $(Y,Z)$.  We define
\[
\cz_{\top}:=\{ Z=\sum_{i=1}^na_iE_i\in\cz\mid Z\neq 0, \mbox{ all }a_i\geq 0, \mbox{ and }Z\cdot E_i\leq 0 \mbox{ for all } i \}.
\]
If there exists $Z\in\cz_{\top}$ such that $Z\cdot Z<0$, then automatically $M_T$ is negative definite \cite[Prop 2(ii)]{Artin}.   In this case, $\cz_{\top}$ admits a unique smallest element $Z_f$, called the \emph{fundamental cycle}.

\begin{thm}\label{tree to R}\cite{Artin, Grauert}.  
Let $T$ denote a labelled tree, with vertex set $\{E_i\mid i\in I\}$, and labels $w_i$.  Suppose that $T$ satisfies the following combinatorial properties.\\
\t{(1)}  $w_i\leq -2$ for all $i\in I$.\\
\t{(2)} There exists $Z\in\cz_{\top}$ such that $Z\cdot Z<0$.\\
\t{(3)} Writing $Z_f$ (which exists by (2)) as $Z_f=\sum_{i\in I}a_iE_i$, then
\[
Z_f\cdot Z_f+\sum_{i\in I}a_i(-w_i-2)=-2.
\]
Then there exists some complete local rational surface singularity $R$, whose minimal resolution has labelled dual graph precisely $T$. 
\end{thm}

A labelled tree satisfying the combinatorial properties in \ref{tree to R} is called a {\em rational tree}.  The above theorem says that every rational tree arises as the labelled dual graph of some complete local rational surface singularity, however this singularity need not be unique.
%
%

We are now ready to prove \ref{uncount G}.
\begin{proof} 
Consider the following labelled trees
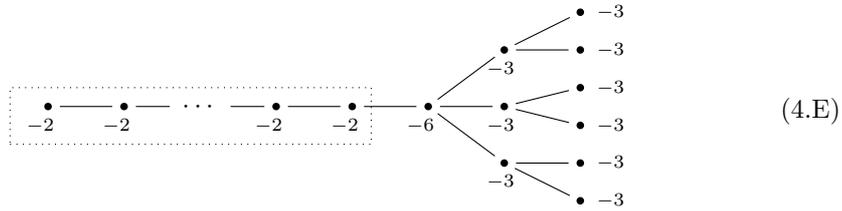
\begin{equation}
\begin{array}{c}
\begin{tikzpicture}[xscale=1]
\node (1) at (1,0) [vertex] {};
\node (2) at (2,0) [vertex] {};
\node (3) at (4,0) [vertex] {};
\node (4) at (5,0)[vertex] {};
\node (5) at (6,0)[vertex] {};
\node (6t) at (7,0.75)[vertex] {};
\node (6m) at (7,0)[vertex] {};
\node (6b) at (7,-0.75)[vertex] {};
\node at (3,0) {$\cdots$};
\node (1a) at (0.9,-0.25) {$\scriptstyle -2$};
\node (2a) at (1.9,-0.25) {$\scriptstyle -2$};
\node (3a) at (3.9,-0.25) {$\scriptstyle -2$};
\node (4a) at (4.9,-0.25) {$\scriptstyle -2$};
\node (5a) at (5.9,-0.25) {$\scriptstyle -6$};
\node (6at) at (6.95,0.5) {$\scriptstyle -3$};
\node (6am) at (6.95,-0.25) {$\scriptstyle -3$};
\node (6ab) at (6.95,-1) {$\scriptstyle -3$};
\node (8a) at (8,1.25)[vertex] {};
\node (8b) at (8,0.75)[vertex] {};
\node (8c) at (8,0.25)[vertex] {};
\node (8d) at (8,-0.25)[vertex] {};
\node (8e) at (8,-0.75)[vertex] {};
\node (8f) at (8,-1.25)[vertex] {};
\node at (8.4,1.25) {$\scriptstyle -3$};
\node at (8.4,0.75) {$\scriptstyle -3$};
\node at (8.4,0.25) {$\scriptstyle -3$};
\node at (8.4,-0.25) {$\scriptstyle -3$};
\node at (8.4,-0.75) {$\scriptstyle -3$};
\node at (8.4,-1.25) {$\scriptstyle -3$};
\draw [-] (1) -- (2);
\draw [-] (2) -- (2.6,0);
\draw [-] (3.4,0) -- (3);
\draw [-] (3) -- (4);
\draw [-] (4) -- (5);
\draw [-] (5) -- (6t);
\draw [-] (5) -- (6m);
\draw [-] (5) -- (6b);
\draw [-] (6t) -- (8a);
\draw [-] (6t) -- (8b);
\draw [-] (6m) -- (8c);
\draw [-] (6m) -- (8d);
\draw [-] (6b) -- (8e);
\draw [-] (6b) -- (8f);
 \draw [dotted] (0.5,0.25) -- (0.5,-0.5) -- (5.25,-0.5) -- (5.25,0.25) -- cycle;
\end{tikzpicture}
\end{array}\label{TypeA}
\end{equation}

\begin{equation}
\begin{array}{c}
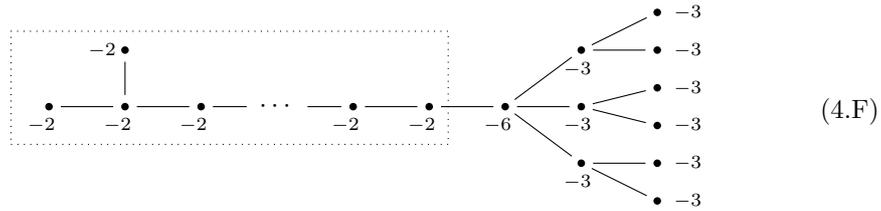
\begin{tikzpicture}
\node (0) at (0,0) [vertex] {};
\node (1) at (1,0) [vertex] {};
\node (1b) at (1,0.75) [vertex] {};
\node (2) at (2,0) [vertex] {};
\node (3) at (4,0) [vertex] {};
\node (4) at (5,0)[vertex] {};
\node (5) at (6,0)[vertex] {};
\node (6t) at (7,0.75)[vertex] {};
\node (6m) at (7,0)[vertex] {};
\node (6b) at (7,-0.75)[vertex] {};
\node at (3,0) {$\cdots$};
\node (0a) at (-0.1,-0.25) {$\scriptstyle - 2$};
\node (1a) at (0.9,-0.25) {$\scriptstyle -2$};
\node (1ba) at (0.7,0.75) {$\scriptstyle - 2$};
\node (2a) at (1.9,-0.25) {$\scriptstyle -2$};
\node (3a) at (3.9,-0.25) {$\scriptstyle -2$};
\node (4a) at (4.9,-0.25) {$\scriptstyle -2$};
\node (5a) at (5.9,-0.25) {$\scriptstyle -6$};
\node (6at) at (6.95,0.5) {$\scriptstyle -3$};
\node (6am) at (6.95,-0.25) {$\scriptstyle -3$};
\node (6ab) at (6.95,-1) {$\scriptstyle -3$};
\node (8a) at (8,1.25)[vertex] {};
\node (8b) at (8,0.75)[vertex] {};
\node (8c) at (8,0.25)[vertex] {};
\node (8d) at (8,-0.25)[vertex] {};
\node (8e) at (8,-0.75)[vertex] {};
\node (8f) at (8,-1.25)[vertex] {};
\node at (8.4,1.25) {$\scriptstyle -3$};
\node at (8.4,0.75) {$\scriptstyle -3$};
\node at (8.4,0.25) {$\scriptstyle -3$};
\node at (8.4,-0.25) {$\scriptstyle -3$};
\node at (8.4,-0.75) {$\scriptstyle -3$};
\node at (8.4,-1.25) {$\scriptstyle -3$};
\draw [-] (0) -- (1);
\draw [-] (1) -- (2);
\draw [-] (2) -- (2.6,0);
\draw [-] (3.4,0) -- (3);
\draw [-] (3) -- (4);
\draw [-] (1) -- (1b);
\draw [-] (4) -- (5);
\draw [-] (5) -- (6t);
\draw [-] (5) -- (6m);
\draw [-] (5) -- (6b);
\draw [-] (6t) -- (8a);
\draw [-] (6t) -- (8b);
\draw [-] (6m) -- (8c);
\draw [-] (6m) -- (8d);
\draw [-] (6b) -- (8e);
\draw [-] (6b) -- (8f);

 \draw [dotted] (-0.5,1) -- (-0.5,-0.5) -- (5.25,-0.5) -- (5.25,1) -- cycle;
\end{tikzpicture}
\end{array}\label{TypeD}
\end{equation}

\begin{equation}
\begin{array}{c}
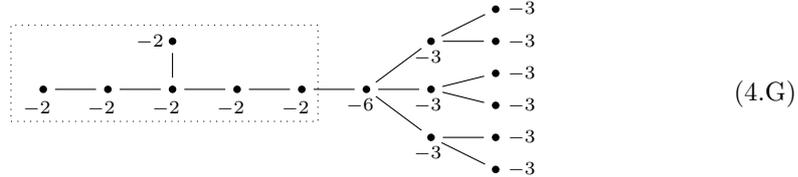
\begin{tikzpicture}[xscale=0.85,yscale=0.85]
\node (0) at (0,0) [vertex] {};
\node (1) at (1,0) [vertex] {};
\node (1b) at (2,0.75) [vertex] {};
\node (2) at (2,0) [vertex] {};
\node (3) at (3,0) [vertex] {};
\node (4) at (4,0) [vertex] {};
\node (5) at (5,0)[vertex] {};
\node (6t) at (6,0.75)[vertex] {};
\node (6m) at (6,0)[vertex] {};
\node (6b) at (6,-0.75)[vertex] {};
\node (7a) at (7,1.25)[vertex] {};
\node (7b) at (7,0.75)[vertex] {};
\node (7c) at (7,0.25)[vertex] {};
\node (7d) at (7,-0.25)[vertex] {};
\node (7e) at (7,-0.75)[vertex] {};
\node (7f) at (7,-1.25)[vertex] {};
\node (0a) at (-0.1,-0.3) {$\scriptstyle -2$};
\node (1a) at (0.9,-0.3) {$\scriptstyle -2$};
\node (1ba) at (1.65,0.75) {$\scriptstyle -2$};
\node (2a) at (1.9,-0.3) {$\scriptstyle -2$};
\node (3a) at (2.9,-0.3) {$\scriptstyle -2$};
\node (4a) at (3.9,-0.3) {$\scriptstyle -2$};
\node (5a) at (4.9,-0.25) {$\scriptstyle -6$};
\node (6at) at (5.95,0.5) {$\scriptstyle -3$};
\node (6am) at (5.95,-0.25) {$\scriptstyle -3$};
\node (6ab) at (5.95,-1) {$\scriptstyle -3$};
\node at (7.4,1.25) {$\scriptstyle -3$};
\node at (7.4,0.75) {$\scriptstyle -3$};
\node at (7.4,0.25) {$\scriptstyle -3$};
\node at (7.4,-0.25) {$\scriptstyle -3$};
\node at (7.4,-0.75) {$\scriptstyle -3$};
\node at (7.4,-1.25) {$\scriptstyle -3$};
\draw [-] (0) -- (1);
\draw [-] (1) -- (2);
\draw [-] (2) -- (3);
\draw [-] (3) -- (4);
\draw [-] (2) -- (1b);
\draw [-] (4) -- (5);
\draw [-] (5) -- (6t);
\draw [-] (5) -- (6m);
\draw [-] (5) -- (6b);
\draw [-] (6t) -- (7a);
\draw [-] (6t) -- (7b);
\draw [-] (6m) -- (7c);
\draw [-] (6m) -- (7d);
\draw [-] (6b) -- (7e);
\draw [-] (6b) -- (7f);
 \draw [dotted] (-0.5,1) -- (-0.5,-0.5) -- (4.25,-0.5) -- (4.25,1) -- cycle;
\end{tikzpicture}
\end{array}\label{TypeE6}
\end{equation}

\begin{equation}
\begin{array}{c}
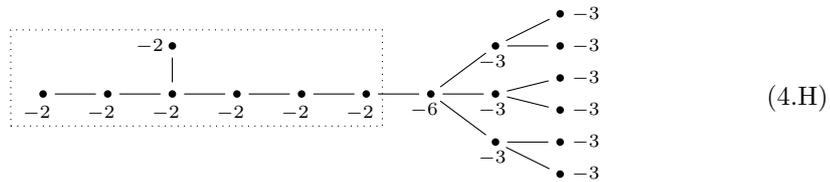
\begin{tikzpicture}[xscale=0.85,yscale=0.85]
\node (0) at (0,0) [vertex] {};
\node (1) at (1,0) [vertex] {};
\node (1b) at (2,0.75) [vertex] {};
\node (2) at (2,0) [vertex] {};
\node (3) at (3,0) [vertex] {};
\node (4) at (4,0) [vertex] {};
\node (5) at (5,0)[vertex] {};
\node (6) at (6,0)[vertex] {};
\node (7t) at (7,0.75)[vertex] {};
\node (7m) at (7,0)[vertex] {};
\node (7b) at (7,-0.75)[vertex] {};
\node (8a) at (8,1.25)[vertex] {};
\node (8b) at (8,0.75)[vertex] {};
\node (8c) at (8,0.25)[vertex] {};
\node (8d) at (8,-0.25)[vertex] {};
\node (8e) at (8,-0.75)[vertex] {};
\node (8f) at (8,-1.25)[vertex] {};
\node (0a) at (-0.1,-0.3) {$\scriptstyle -2$};
\node (1a) at (0.9,-0.3) {$\scriptstyle -2$};
\node (1ba) at (1.65,0.75) {$\scriptstyle -2$};
\node (2a) at (1.9,-0.3) {$\scriptstyle -2$};
\node (3a) at (2.9,-0.3) {$\scriptstyle -2$};
\node (4a) at (3.9,-0.3) {$\scriptstyle -2$};
\node (5a) at (4.9,-0.3) {$\scriptstyle -2$};
\node (6a) at (5.9,-0.25) {$\scriptstyle -6$};
\node (7at) at (6.95,0.5) {$\scriptstyle -3$};
\node (7am) at (6.95,-0.25) {$\scriptstyle -3$};
\node (7ab) at (6.95,-1) {$\scriptstyle -3$};
\node at (8.4,1.25) {$\scriptstyle -3$};
\node at (8.4,0.75) {$\scriptstyle -3$};
\node at (8.4,0.25) {$\scriptstyle -3$};
\node at (8.4,-0.25) {$\scriptstyle -3$};
\node at (8.4,-0.75) {$\scriptstyle -3$};
\node at (8.4,-1.25) {$\scriptstyle -3$};
\draw [-] (0) -- (1);
\draw [-] (1) -- (2);
\draw [-] (2) -- (3);
\draw [-] (3) -- (4);
\draw [-] (2) -- (1b);
\draw [-] (4) -- (5);
\draw [-] (5) -- (6);
\draw [-] (6) -- (7t);
\draw [-] (6) -- (7m);
\draw [-] (6) -- (7b);
\draw [-] (7t) -- (8a);
\draw [-] (7t) -- (8b);
\draw [-] (7m) -- (8c);
\draw [-] (7m) -- (8d);
\draw [-] (7b) -- (8e);
\draw [-] (7b) -- (8f);

 \draw [dotted] (-0.5,1) -- (-0.5,-0.5) -- (5.25,-0.5) -- (5.25,1) -- cycle;
\end{tikzpicture}
\end{array}\label{TypeE7}
\end{equation}
It is an easy combinatorial check to show that each labelled graph above satisfies the criteria in \ref{tree to R}, so consequently there is a (not necessarily unique) complete rational surface singularity corresponding to each.  We do this for (\ref{TypeD}), the rest being similar.  Labelling the vertices in (\ref{TypeD}) by
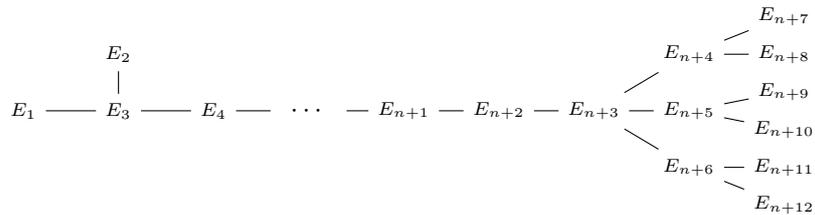
\[
\begin{array}{c}
\begin{tikzpicture}[xscale=1.25]
\node (0) at (0,0) {$\scriptstyle E_1$};
\node (1) at (1,0) {$\scriptstyle E_3$};
\node (1b) at (1,0.75) {$\scriptstyle E_2$};
\node (2) at (2,0) {$\scriptstyle E_4$};
\node (3) at (4,0) {$\scriptstyle E_{n+1}$};
\node (4) at (5,0) {$\scriptstyle E_{n+2}$};
\node (5) at (6,0) {$\scriptstyle E_{n+3}$};
\node (6t) at (7,0.75) {$\scriptstyle E_{n+4}$};
\node (6m) at (7,0) {$\scriptstyle E_{n+5}$};
\node (6b) at (7,-0.75) {$\scriptstyle E_{n+6}$};
\node (8a) at (8,1.25) {$\scriptstyle E_{n+7}$};
\node (8b) at (8,0.75) {$\scriptstyle E_{n+8}$};
\node (8c) at (8,0.25) {$\scriptstyle E_{n+9}$};
\node (8d) at (8,-0.25) {$\scriptstyle E_{n+10}$};
\node (8e) at (8,-0.75) {$\scriptstyle E_{n+11}$};
\node (8f) at (8,-1.25) {$\scriptstyle E_{n+12}$};
\node at (3,0) {$\cdots$};
\draw [-] (0) -- (1);
\draw [-] (1) -- (2);
\draw [-] (2) -- (2.6,0);
\draw [-] (3.4,0) -- (3);
\draw [-] (3) -- (4);
\draw [-] (1) -- (1b);
\draw [-] (4) -- (5);
\draw [-] (5) -- (6t);
\draw [-] (5) -- (6m);
\draw [-] (5) -- (6b);
\draw [-] (6t) -- (8a);
\draw [-] (6t) -- (8b);
\draw [-] (6m) -- (8c);
\draw [-] (6m) -- (8d);
\draw [-] (6b) -- (8e);
\draw [-] (6b) -- (8f);
\end{tikzpicture}
\end{array}
\]
then it is easy to see that $Z:=\sum_{i=1}^{2}E_i+\sum_{i=3}^{n+2}2E_i+\sum_{i=n+3}^{n+12}E_i$ satisfies $Z\cdot E_i\leq 0$ for all $1\leq i\leq n+12$, hence $Z\in\cz_{\top}$.  We denote $Z$ as
\[
Z=
\begin{array}{c}
\begin{tikzpicture}
\node (0) at (0,0) {$\scriptstyle 1$};
\node (1) at (1,0) {$\scriptstyle 2$};
\node (1b) at (1,0.75) {$\scriptstyle 1$};
\node (2) at (2,0) {$\scriptstyle 2$};
\node (3) at (4,0) {$\scriptstyle 2$};
\node (4) at (5,0) {$\scriptstyle 2$};
\node (5) at (6,0) {$\scriptstyle 1$};
\node (6t) at (7,0.75) {$\scriptstyle 1$};
\node (6m) at (7,0) {$\scriptstyle 1$};
\node (6b) at (7,-0.75) {$\scriptstyle 1$};
\node (8a) at (8,1.25) {$\scriptstyle 1$};
\node (8b) at (8,0.75) {$\scriptstyle 1$};
\node (8c) at (8,0.25) {$\scriptstyle 1$};
\node (8d) at (8,-0.25) {$\scriptstyle 1$};
\node (8e) at (8,-0.75) {$\scriptstyle 1$};
\node (8f) at (8,-1.25) {$\scriptstyle 1$};
\node at (3,0) {$\cdots$};
\draw [-] (0) -- (1);
\draw [-] (1) -- (2);
\draw [-] (2) -- (2.6,0);
\draw [-] (3.4,0) -- (3);
\draw [-] (3) -- (4);
\draw [-] (1) -- (1b);
\draw [-] (4) -- (5);
\draw [-] (5) -- (6t);
\draw [-] (5) -- (6m);
\draw [-] (5) -- (6b);
\draw [-] (6t) -- (8a);
\draw [-] (6t) -- (8b);
\draw [-] (6m) -- (8c);
\draw [-] (6m) -- (8d);
\draw [-] (6b) -- (8e);
\draw [-] (6b) -- (8f);
\end{tikzpicture}
\end{array}
\]
From this we see that
\[
(Z\cdot E_i)_{i=1}^{n+12}=
\begin{array}{c}
\begin{tikzpicture}
\node (0) at (0,0) {$\scriptstyle 0$};
\node (1) at (1,0) {$\scriptstyle 0$};
\node (1b) at (1,0.75) {$\scriptstyle 0$};
\node (2) at (2,0) {$\scriptstyle 0$};
\node (3) at (4,0) {$\scriptstyle 0$};
\node (4) at (5,0) {$\scriptstyle -1$};
\node (5) at (6,0) {$\scriptstyle -1$};
\node (6t) at (7,0.75) {$\scriptstyle 0$};
\node (6m) at (7,0) {$\scriptstyle 0$};
\node (6b) at (7,-0.75) {$\scriptstyle 0$};
\node (8a) at (8,1.25) {$\scriptstyle -2$};
\node (8b) at (8,0.75) {$\scriptstyle -2$};
\node (8c) at (8,0.25) {$\scriptstyle -2$};
\node (8d) at (8,-0.25) {$\scriptstyle -2$};
\node (8e) at (8,-0.75) {$\scriptstyle -2$};
\node (8f) at (8,-1.25) {$\scriptstyle -2$};
\node at (3,0) {$\cdots$};
\draw [-] (0) -- (1);
\draw [-] (1) -- (2);
\draw [-] (2) -- (2.6,0);
\draw [-] (3.4,0) -- (3);
\draw [-] (3) -- (4);
\draw [-] (1) -- (1b);
\draw [-] (4) -- (5);
\draw [-] (5) -- (6t);
\draw [-] (5) -- (6m);
\draw [-] (5) -- (6b);
\draw [-] (6t) -- (8a);
\draw [-] (6t) -- (8b);
\draw [-] (6m) -- (8c);
\draw [-] (6m) -- (8d);
\draw [-] (6b) -- (8e);
\draw [-] (6b) -- (8f);
\end{tikzpicture}
\end{array}
\]
so $Z\in\cz_{\top}$ and $Z\cdot Z=Z\cdot(\sum_{i=1}^{2}E_i+\sum_{i=3}^{n+2}2E_i+\sum_{i=n+3}^{n+12}E_i)=0+2(-1)+(-1-2-2-2-2-2-2)=-15$.   Hence condition (2) in \ref{tree to R} is satisfied.  For condition (3), by the standard Laufer algorithm, $Z_f=Z$, so $Z_f\cdot Z_f=-15$. On the other hand $\sum_{i\in I}a_i(-E_i^2-2)=4+1+1+1+1+1+1+1+1+1=13$, so  $Z_f\cdot Z_f+\sum_{i\in I}a_i(-E_i^2-2)=-15+13=-2$, as required.

Now in the above diagrams, for clarity we have drawn a box around the curves that get contracted to form $X^{\cc}$.  Hence a $\Gamma=\End_R(N^{\cc})$ corresponding to (\ref{TypeA}) has the GP finite type corresponding to cyclic groups,  by \ref{main triangles} applied to $\End_R(N^{\cc})$.  Similarly, a $\Gamma$ corresponding to (\ref{TypeD}) has the GP finite type corresponding to binary dihedral groups, (\ref{TypeE6}) corresponds to binary tetrahedral groups, and (\ref{TypeE6}) corresponds to binary octahedral groups.

Now each of the above trees has more than one vertex that meets precisely three edges, so by the classification \cite[\S1 p2]{Laufer} they are not pseudo--taut, and further in each of the above trees there exists a vertex that meets precisely four edges, so by the classification \cite[\S2 p2]{Laufer} they are not taut.  This means that in \ref{tree to R} there are uncountably many (non-isomorphic) $R$ corresponding to each of the above labelled trees.  For each such $R$ we thus obtain an Iwanaga--Gorenstein ring $\End_R(N^{\cc})$ with the desired properties, and further if $R$ and $R^\prime$ both correspond to the same labelled graph, but $R\ncong R^\prime$, then $\End_{R}(N^{\cc})\ncong\End_{R^\prime}(N^{\cc})$ since the centers of $\End_R(N^{\cc})$ and $\End_{R^\prime}(N^{\cc})$ are $R$ and $R^\prime$ respectively. Hence, since there are uncountably many such $R$, there are uncountably many such Iwanaga--Gorenstein rings.
\end{proof}
We give, in \S\ref{geom examples}, some explicit examples illustrating \ref{uncount G} in the case $G=\mathbb{Z}_2$ .

\begin{remark}
We remark that the method in the above proof cannot be applied to $E_8$, since it is well-known that the rational tree $E_8$ (labelled with $-2$'s) cannot be a (strict) subtree of any rational tree \cite[3.11]{TT}.
\end{remark}

\section{Relationship to relative singularity categories}\label{ss:relative-sing-cat}

In the notation of \S\ref{geometry}, let $Y \stackrel{f^{\cs}} \longrightarrow X^\cs \stackrel{g^\cs} \longrightarrow \Spec R$ be a factorization of the minimal resolution of a rational surface singularity, with $\cs \subseteq I$. Let $\Lambda$ be the reconstruction algebra of $R$ and $e\in \Lambda$ be the idempotent corresponding to the identity endomorphism of the special Cohen--Macaulay R-module $N^\cs=R \oplus(\bigoplus_{i \in I \setminus S} M_{i})$.

\begin{defin}
\t{(1)} A triangle functor $Q\colon \cc \to \cd$ is called a \emph{quotient functor} if the induced functor $\cc/\ker Q \to \cd$ is a triangle equivalence. Here $\ker Q \subseteq \cc$ denotes the full subcategory of objects $X$ such that $Q(X)=0$.\\
\t{(2)} A sequence of triangulated categories and triangle functors $\cu \stackrel{F}\ra \ct \stackrel{G}\ra \cq$ is called \emph{exact} if $G$ is a quotient functor with kernel $\cu$, and $F$ is the natural inclusion. 
\end{defin}

In this section, we extend triangle equivalences from \ref{main triangles} to exact sequences of triangulated categories. In particular, this yields triangle equivalences between the relative singularity categories studied in \cite{BK, KY, Kalck}. 
\begin{prop} 
There exists  a commutative diagram of triangulated categories and functors such that the horizontal arrows are equivalences and the columns are exact.
\begin{align}\label{E:Commutative}
\begin{array}{c}
{\SelectTips{cm}{10}
\xy0;/r.4pc/:
(-10,20)*+{ \thick\left( \bigoplus_{i \in \cs} \co_{E_{i}}(-1)\right)}="A2",
(25,20)*+{\thick(\mod\Lambda/\Lambda e\Lambda) }="A3",
(-10,10)*+{\frac{ \displaystyle \Db(\coh Y)}{\displaystyle \thick\bigl(\co_{Y} \oplus\bigl( {\textstyle \bigoplus}_{i \in I \setminus \cs} \cm_{i}\bigr)\bigr)}}="a2",
(25,10)*+{\frac{\displaystyle \Db(\mod\Lambda)}{\displaystyle \thick(\Lambda e)}}="a3",
(-10,0)*+{\Dsg(X^\cs)}="b2",
(25,0)*+{ \Dsg(e\Lambda e)}="b3",
\ar"A2";"A3"^{\sim}
\ar@{^{(}->}"A2";"a2"
\ar@{^{(}->}"A3";"a3"
\ar@{->>}"a2";"b2"_{\Rfcs}
\ar@{->>}"a3";"b3"^{e(-)}
\ar"b2";"b3"^{\RHom_{X^\cs}(\cv_{\cs},-)}_{\sim}
\ar"a2";"a3"^{\RHom_Y(\cv_\emptyset,-)}_{\sim}
\endxy}
\end{array}
\end{align}
By an abuse of notation, the induced triangle functors in the lower square are labelled by the inducing triangle functors from the diagram in \ref{main Db}.
\end{prop}
\begin{proof}
We start with the lower square. Since the corresponding diagram in \ref{main Db} commutes, it suffices to show that the induced functors above are well-defined. Clearly, the equivalence $\RHom_{Y}(\cv_{\emptyset}, -)$ from \ref{main Db} maps $ \co_{Y} \oplus(\bigoplus_{i \in I \setminus \cs} \cm_{i})$ to $\Lambda e$. Hence, it induces an equivalence on the triangulated quotient categories. Since $\RHom_{X^\cs}(\cv_{\cs}, -)$ is an equivalence by \ref{main Db} and the subcategories $\Perf(X^\cs)$ respectively $\Perf(e\Lambda e)$ can be defined intrinsically, we get a well-defined equivalence on the bottom of diagram (\ref{E:Commutative}). The functor on the right is a well-defined quotient functor by \ref{P:OneIdempotent}. Now, the functor on the left is a well-defined quotient functor by the commutativity of the diagram in \ref{main Db} and the considerations above. 

The category $\thick(\mod\Lambda/\Lambda e\Lambda)$ is the kernel of the quotient functor $e(-)$, by~\ref{P:OneIdempotent}.  Since $R$ has isolated singularities, the algebra $\Lambda/\Lambda e \Lambda$ is always finite dimensional and so $\thick(\mod\Lambda/\Lambda e \Lambda) = \thick\left(\bigoplus_{i \in \cs}S_{i} \right)$,
where $S_{i}$ denotes the simple $\Lambda$-module corresponding to the vertex $i$ in the quiver of $\Lambda$.  But under the derived equivalence $\RHom_{Y}(\cv_{\emptyset}, -)$, $S_i$ corresponds to $\co_{E_i}(-1)[1]$ \cite[3.5.7]{VdB1d}, so it follows that we can identify the subcategory $\thick(\mod\Lambda/\Lambda e \Lambda)=\thick\left(\bigoplus_{i \in \cs}S_{i} \right)$ with $\thick\left( \bigoplus_{i \in \cs} \co_{E_{i}}(-1)\right)$, inducing the top half of the diagram.
\end{proof}
\begin{remark}
The functor $\RHom_{X^\cs}(\cv_{\cs}, -)$ identifies $\Perf(X^\cs)$ with $\Perf(e\Lambda e) \cong \thick(\Lambda e) \subseteq \Db(\mod \Lambda)$. Hence, applying the quasi-inverse of $\RHom_{Y}(\cv_{\emptyset}, -)$ to $\thick(\Lambda e)$ yields a triangle equivalence
$\Perf(X^\cs) \cong \thick\bigl(\co_{Y}\oplus(\bigoplus_{i \in I \setminus \cs} \cm_{i})\bigr)$. In particular, there is an equivalence
\begin{align}\label{E:RemarkRelSingCat}
\frac{\displaystyle \Db(\coh Y)}{\displaystyle \Perf(X^\cs)} \stackrel{\sim}\longrightarrow \frac{\displaystyle \Db(\mod\Lambda)}{\displaystyle \thick(\Lambda e)}. 
\end{align}
Analysing the commutative diagram in \ref{main Db} shows that $\Perf(X^\cs) \cong \thick\bigl(\co_{Y}\oplus(\bigoplus_{i \in I \setminus \cs} \cm_{i})\bigr)$ is obtained as a restriction of $\mathbf{L}(f^\cs)^{*}$.
\end{remark}

\medskip

If we contract only ($-2$)-curves (i.e.~if $\cs \subseteq \cc$ holds), then we know that $\Dsg(X^\cs)$ splits into a direct sum of singularity categories of ADE--surface singularities (\ref{main triangles}).  In this case, it turns out that the diagram above admits an extension to the right and that in fact all the triangulated categories in our (extended) diagram split into blocks indexed by the singularities of the Gorenstein scheme $X^\cs$.

Let us fix some notation. For a singular point $x \in \Sing X^\cs$ let $R_{x}=\widehat{\co}_{X^\cs, x}$, and let $f_{x}\colon Y_{x} \ra \Spec R_{x} $ be the minimal resolution of singularities.

\begin{prop} Assume $\cs\subseteq\cc$.
There exists  a commutative diagram of triangulated categories and functors such that the horizontal arrows are equivalences and the columns are exact.
\begin{align}\label{E:extension-to-right}
\begin{array}{c}
{\SelectTips{cm}{10}
\xy0;/r.4pc/:
(-10,20)*+{\thick(\mod\Lambda/\Lambda e\Lambda)}="A2",
(20,20)*+{ \bigoplus\limits_{x \in \Sing X^\cs} \ker(\mathbf{R} (f_{x})_{*}) }="A3",
(-10,10)*+{\frac{\displaystyle \Db(\mod\Lambda)}{\displaystyle \thick(\Lambda e)}}="a2",
(20,10)*+{ \bigoplus\limits_{x \in \Sing X^\cs} \frac{\displaystyle \Db(\coh Y_{x})}{\displaystyle \Perf(R_{x})} }="a3",
(-10,0)*+{\Dsg(e\Lambda e)}="b2",
(20,0)*+{  \bigoplus\limits_{x \in \Sing X^\cs} \Dsg(R_{x})}="b3",
\ar"A2";"A3"^{\sim}
\ar@{^{(}->}"A2";"a2"
\ar@{^{(}->}"A3";"a3"
\ar@{->>}"a3";"b3"^{\bigoplus_{x \in \Sing X^\cs}\mathbf{R} (f_{x})_{*}}
\ar@{->>}"a2";"b2"_{e(-)}
\ar"b2";"b3"^{\sim}
\ar"a2";"a3"^{\sim}
\endxy}
\end{array}
\end{align}
\end{prop}
\begin{proof}
We need some preparation. Note that by the derived McKay correspondence \cite{KapranovVasserot, BKR}, there are derived equivalences 
$\Db(\coh Y_{x}) \ra \Db(\mod\Pi_{x})$, where $\Pi_{x}$ is the Auslander algebra of the Frobenius category of maximal Cohen--Macaulay $R_{x}$-modules $\CM(R_{x})$.   Now we have two Frobenius categories $\ce_{1}:=\SCM_{N^\cs}(R)$ and $\ce_{2}:=\bigoplus_{x \in \Sing X^\cs} \CM(R_{x})$, which clearly satisfy the conditions (FM1)--(FM4) in \cite[Subsection 4.4]{KY} and whose stable categories are Hom-finite and idempotent complete. Further, $\ce_1$ and $\ce_2$ are stably equivalent by \ref{main triangles}.

Now, by \cite[4.7(b)]{KY} there are triangle equivalences  
\begin{align}
\Db(\mod\Lambda)/\thick(\Lambda e) \cong \mathrm{per}\bigl(\Lambda_{dg}(\underline{\ce}_{1})\bigr)\label{E:Tria1}\\
\bigoplus_{x \in \Sing X^\cs} \Db(\mod \Pi_{x})/\Perf(R_{x}) \cong \mathrm{per}\bigl(\Lambda_{dg}(\underline{\ce}_{2})\bigr)\label{E:Tria2}
\end{align}
where by definition $\Lambda_{dg}(\underline{\ce}_{1})$ and $\Lambda_{dg}(\underline{\ce}_{2})$ are dg algebras that depend only on (the triangulated structure of) the stable Frobenius categories $\underline{\ce}_{1}$ and $\underline{\ce}_{2}$  (the quotient category $\Db(\mod\Lambda)/\thick(\Lambda e)$ is idempotent complete by \cite[2.69(a)]{Kalck} combined with \ref{glrecon3} and the completeness of $R$). Hence, since $\ce_1$ and $\ce_2$ are stably equivalent, these two dg algebras are isomorphic.  Thus the combination of the equivalences \eqref{E:Tria1} and \eqref{E:Tria2} yields a triangle equivalence
\begin{align}\label{E:RelSgCat}
\frac{\displaystyle \Db(\mod\Lambda)}{\displaystyle \thick(\Lambda e)}  \longrightarrow \bigoplus_{x \in \Sing X^\cs} \frac{\displaystyle \Db\bigl(\mod \Pi_{x} \bigr)}{\displaystyle \Perf(R_{x})}
\end{align}
which, in conjunction with the derived McKay Correspondence,  yields the equivalence of triangulated categories in the middle of \eqref{E:extension-to-right}.

Furthermore, the functors $\Hom_{\Lambda}(\Lambda e, -)$ and $\bigoplus_{x \in \Sing X^\cs} \mathbf{R} (f_{x})_{*}$ are quotient functors with kernels $\thick(\mod\Lambda/\Lambda e\Lambda)$ and $\bigoplus_{x \in \Sing X^\cs} \ker(\mathbf{R} (f_{x})_{*})$, respectively. These subcategories admit intrinsic descriptions (c.f.\ \cite[Corollary 6.17]{KY}). Hence, there is an induced equivalence, which renders the upper square commutative. This in turn induces an equivalence on the bottom of (\ref{E:extension-to-right}), such that the lower square commutes.
\end{proof}

\begin{remark}
Using \eqref{E:RemarkRelSingCat} together with an appropriate adaption of the techniques developed in \cite{BK} may yield a more direct explanation for the block decomposition in \eqref{E:extension-to-right}. 
\end{remark}

\section{Examples}\label{ss:examples}

In this section we illustrate some of the previous results with some examples.  Our construction in \S\ref{ss:alg-tria-cl-sing} relies on finding some $M$ such that $\gl\End_\Lambda(\Lambda\oplus M)<\infty$, so we give explicit examples of when this occurs in both finite dimensional algebras, and in geometry. 

\subsection{Iwanaga--Gorenstein rings of finite GP type}\label{geom examples} As a special case of \ref{uncount G}, there are uncountably many Iwanaga--Gorenstein rings $\Gamma$ with the property that $\underline{\GP}(\Gamma)\simeq\uCM(\mathbb{C}[[x,y]]^{\frac{1}{2}(1,1)})$.  This category has only one indecomposable object, and is the simplest possible triangulated category. Here we show that the abstract setting in  \ref{uncount G} can be used to give explicit examples of such $\Gamma$, presented as a quiver with relations.

\begin{defin}
For all $n\geq 3$, we define the algebra $\Lambda_n$ to be the path algebra of the following quiver
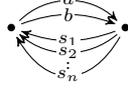
\[
\begin{array}{c}
\begin{tikzpicture} [bend angle=45, looseness=1]
\node (C1) at (0,0) [vertex] {};
\node (C2) at (1.5,0)  [vertex] {};
\draw [->,bend left] ($(C1)+(40:4pt)$) to node[gap]  {$\scriptstyle a$} ($(C2)+(140:4pt)$);
\draw [->,bend left=20,looseness=1] ($(C1)+(20:4pt)$) to node[gap]  {$\scriptstyle b$} ($(C2)+(160:4pt)$);
\draw [->,bend left=20] ($(C2)+(-160:4pt)$) to node[gap]  {$\scriptstyle s_1$} ($(C1)+(-20:4pt)$);
\draw [->,bend left=40] ($(C2)+(-140:4pt)$) to node[gap]  {$\scriptstyle s_2$} ($(C1)+(-40:4pt)$);
\draw [->,bend left=60,looseness=1.5] ($(C2)+(-120:4pt)$) to node[gap]  {$\scriptstyle s_n$} ($(C1)+(-60:4pt)$);
\node at (0.75,-0.47) {$\scriptstyle .$};
\node at (0.75,-0.53) {$\scriptstyle .$};
\end{tikzpicture}
\end{array}
\] 
(where there are $n$ arrows from right to left), subject to the relations
\[
\begin{array}{l}
\begin{array}{c}
s_{n-1}bs_n=s_nbs_{n-1}\\
as_n=(bs_{n-1})^2\\
s_na=(s_{n-1}b)^2
\end{array}\\
\left. \begin{array}{l} as_{i+1}=bs_i\\ s_{i+1}a=s_ib\end{array}\right\} \mbox{ for all } 1\leq i\leq n-2.
\end{array}
\]
\end{defin}
Our main result (\ref{main6.1}) shows that for all $n\geq 3$ the completion $\widehat{\Lambda}_n$ is an Iwanaga--Gorenstein ring with $\id\widehat{\Lambda}_n=3$, such that $\underline{\GP}(\widehat{\Lambda}_n)\simeq \uCM(\mathbb{C}[[x,y]]^{\frac{1}{2}(1,1)})$.  Before we can prove this, we need some notation.  
Let $n\geq 3$, set $m:=2n-1$ and consider the group 
\[
\frac{1}{m}(1,2):=\left\langle \begin{pmatrix} \e_{m}&0\\ 0 &\e_{m}^2 \end{pmatrix}\right\rangle
\]
where $\e_{m}$ is a primitive $m$th root of unity.  The invariants $\mathbb{C}[x,y]^{\frac{1}{m}(1,2)}$ are known to be generated by 
\[
a:=x^{m},\ b_1:=x^{m-2}y,\ b_2:=x^{m-4}y^2,\hdots,\ b_{n-1}:=xy^{n-1},\ c:=y^m
\]
which abstractly as a commutative ring is $\mathbb{C}[a,b_1,\hdots,b_{n-1},c]$ factored by the relations given by the $2\times 2$ minors of the matrix
\[
\begin{pmatrix}
a&b_1 &b_2&\hdots&b_{n-2}&b_{n-1}^2\\
b_1&b_2 &b_3&\hdots&b_{n-1}&c
\end{pmatrix}.
\]
We denote this (non-complete) commutative ring by $R$.  This singularity is toric, and the minimal resolution of $\Spec R$ is well-known to have dual graph
\[
\begin{tikzpicture}
\node (A) at (0,0.25) {};
\node (A) at (0,0) [vertex] {};
\node (B) at (1,0) [vertex] {};
\draw (A) -- (B);
\node (A) at (0,-0.25) {$\scriptstyle -n$};
\node (B) at (1,-0.25) {$\scriptstyle -2$};
\end{tikzpicture}
\]

\begin{thm}\label{main6.1}
Let $n\geq 3$, set $m:=2n-1$ and consider $G:=\frac{1}{m}(1,2)$.  Denote $R:=\mathbb{C}[x,y]^{G}$, presented as $\mathbb{C}[a,b_1,\hdots,b_{n-1},c]/(2\times 2 \mbox{ minors})$ as above.  Then\\
\t{(1)} The $R$-ideal $(a,b_1)$ is the non-free special CM $R$-module corresponding to the ($-n$)-curve in the minimal resolution of $\Spec R$.\\
\t{(2)} $\Lambda_n\cong\End_R(R\oplus (a,b_1))$.\\
In particular, by completing both sides of \t{(2)}, $\widehat{\Lambda}_n$ is an Iwanaga--Gorenstein ring with $\id\widehat{\Lambda}_n=3$, such that $\underline{\GP}(\widehat{\Lambda}_n)\simeq\uCM(\mathbb{C}[[x,y]]^{\frac{1}{2}(1,1)})$.  Further $\widehat{\Lambda}_{n^\prime}\ncong\widehat{\Lambda}_n$ whenever $n^\prime\neq n$.
\end{thm}
\begin{proof}
(1) Let $\rho_0,\hdots, \rho_{m-1}$ be the irreducible representations of $G\cong\mathbb{Z}_m$ over $\mathbb{C}$.  Since $R=\mathbb{C}[x,y]^{G}$, we can consider the CM $R$ modules $S_i:=(\mathbb{C}[x,y]\otimes_\mathbb{C}\rho_i)^{G}$.  It is a well known result of Wunram \cite{Wun_cyclic} that the special CM $R$-modules in this case are $R=S_0$, $S_1$ and $S_2$, with $S_2$ corresponding to the ($-n$)-curve.  We remark that Wunram proved this result under the assumption that $R$ is complete, but the result is still true in the non-complete case \cite{Craw,ReconA}.  Further, $S_2$ is generated by $x^2,y$ as an $R$-module \cite{Wun_cyclic}.  It is easy to check that under the new coordinates, $S_2$ is isomorphic to $(a,b_1)$.\\
(2) We prove this using key varieties.  \\
\noindent
{\em Step 1.} Consider the commutative ring $\mathbb{C}[a,b_{1}^{(1)},b_{1}^{(2)},\hdots,b_{n-1}^{(1)},b_{n-1}^{(2)},c]$ factored by the relations given by the $2\times 2$ minors of the matrix
\[
\begin{pmatrix}
a&b_{1}^{(1)} &b_{2}^{(1)}&\hdots&b_{n-2}^{(1)}&b_{n-1}^{(1)}\\
b_{1}^{(2)}&b_{2}^{(2)} &b_{3}^{(2)}&\hdots&b_{n-1}^{(2)}&c
\end{pmatrix}.
\]
We denote this factor ring by $S$.  We regard $\Spec S$ as a key variety which we then cut (in Step 4) to obtain our ring $R$.

\medskip
\noindent
{\em Step 2.} We blowup the ideal $(a,b_{1}^{(2)})$ of $S$ to give a variety, denoted $Y$, covered by the two affine opens
\[
\mathbb{C}[b_{1}^{(2)},b_{2}^{(2)},\hdots,b_{n-1}^{(2)},c,\tfrac{a}{b_{1}^{(2)}}]\qquad \mathbb{C}[a,b_{1}^{(1)},b_{2}^{(1)},\hdots,b_{n-1}^{(1)},\tfrac{b_{1}^{(2)}}{a}]
\]
The resulting map $f:Y\to\Spec S$ has fibres at most one-dimensional, so we know from \cite{VdB1d} that $Y$ has a tilting bundle. Using the above explicit open cover and morphism, there is an ample line bundle $\cl$ on $Y$ generated by global sections, satisfying $\cl\cdot E=1$ (where $E$ is the $\mathbb{P}^1$ above the origin), with the property that $H^1(\cl^\vee)=0$.  This means, by \cite[3.2.5]{VdB1d}, that $\cv:=\co\oplus\cl$ is a tilting bundle. As is always true in the one-dimensional fibre tilting setting (where $f$ is projective birational between integral normal schemes), $\End_Y(\co\oplus\cl)\cong\End_S(S\oplus f_*\cl)$.  In the explicit construction of $Y$ above, it is clear that $f_*\cl=(a,b_1^{(2)})$.  This shows that $\End_S(S\oplus (a,b_1^{(2)}))$ is derived equivalent to $Y$.

\medskip
\noindent
{\em Step 3.}  We present $\End_S(S\oplus (a,b_1^{(2)}))$ as a quiver with relations.  This is easy, since $Y$ is smooth. We have
\[
\End_S(S\oplus (a,b_1^{(2)}))\cong\left(\begin{array}{cc} S& (a,b_1^{(2)})\\
 (a,b_1^{(2)})^* &S  \end{array}\right),
\]
and we can check that all generators can be seen on the diagram 
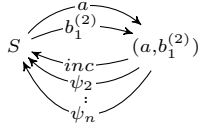
\[
\begin{array}{c}
\begin{tikzpicture} [bend angle=45, looseness=1]
\node (C1) at (0,0)  {$\scriptstyle S$};
\node (C2) at (2,0) {$\scriptstyle (a,b_1^{(2)})$};
\draw [->,bend left] ($(C1)+(40:7pt)$) to node[gap]  {$\scriptstyle a$} ($(C2)+(140:12pt)$);
\draw [->,bend left=20,looseness=1] ($(C1)+(20:7pt)$) to node[gap]  {$\scriptstyle b_1^{(2)}$} ($(C2)+(160:13pt)$);
\draw [->,bend left=20] ($(C2)+(-160:15pt)$) to node[gap]  {$\scriptstyle inc$} ($(C1)+(-20:7pt)$);
\draw [->,bend left=40] ($(C2)+(-140:13pt)$) to node[gap]  {$\scriptstyle \psi_2$} ($(C1)+(-40:7pt)$);
\draw [->,bend left=60,looseness=1.5] ($(C2)+(-120:11pt)$) to node[gap]  {$\scriptstyle \psi_n$} ($(C1)+(-60:7pt)$);
\node at (0.95,-0.7) {$\scriptstyle .$};
\node at (0.95,-0.75) {$\scriptstyle .$};
\end{tikzpicture}
\end{array}
\] 
where $\psi_i:=\frac{b_{i-1}^{(1)}}{a}=\frac{b_i^{(2)}}{b_1^{(2)}}$ for all $2\leq i\leq n-1$, and $\psi_n:=\frac{b_{n-1}^{(1)}}{a}=\frac{c}{b_1^{(2)}}$.  Thus if we consider the quiver $Q$ 
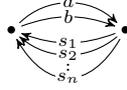
\[
\begin{array}{c}
\begin{tikzpicture} [bend angle=45, looseness=1]
\node (C1) at (0,0) [vertex] {};
\node (C2) at (1.5,0)  [vertex] {};
\draw [->,bend left] ($(C1)+(40:4pt)$) to node[gap]  {$\scriptstyle a$} ($(C2)+(140:4pt)$);
\draw [->,bend left=20,looseness=1] ($(C1)+(20:4pt)$) to node[gap]  {$\scriptstyle b$} ($(C2)+(160:4pt)$);
\draw [->,bend left=20] ($(C2)+(-160:4pt)$) to node[gap]  {$\scriptstyle s_1$} ($(C1)+(-20:4pt)$);
\draw [->,bend left=40] ($(C2)+(-140:4pt)$) to node[gap]  {$\scriptstyle s_2$} ($(C1)+(-40:4pt)$);
\draw [->,bend left=60,looseness=1.5] ($(C2)+(-120:4pt)$) to node[gap]  {$\scriptstyle s_n$} ($(C1)+(-60:4pt)$);
\node at (0.75,-0.47) {$\scriptstyle .$};
\node at (0.75,-0.53) {$\scriptstyle .$};
\end{tikzpicture}
\end{array}
\] 
with relations $\mathcal{R}$
\[
\begin{array}{cl}
as_{i}b=bs_ia&\mbox{for all }1\leq i\leq n\\
s_ias_j=s_jas_i&\mbox{for all } 1\leq i<j\leq n.
\end{array}
\]
then there is a natural surjective ring homomorphism
\[
\mathbb{C}Q/\mathcal{R}\to\End_S(S\oplus (a,b_1^{(2)})).
\]
But everything above is graded (with arrows all having grade one), and so a Hilbert series calculation shows that the above ring homomorphism must also be bijective.

\medskip
\noindent
{\em Step 4.}  We base change, and show that we can add central relations to the presentation of $\End_S(S\oplus (a,b_1^{(2)}))$ in Step 3 to obtain a presentation for $\End_R(R\oplus(a,b_1))$.

Factoring $S$ by the regular element $b_1^{(1)}-b_1^{(2)}$ we obtain a ring denoted $R_1$.  Factoring $R_1$ by the regular element $b_2^{(1)}-b_2^{(2)}$ we obtain a ring denoted $R_2$.  Continuing in this manner, factor $R_{n-3}$ by $b_{n-2}^{(1)}-b_{n-2}^{(2)}$ to obtain $R_{n-2}$.  Finally, factor $R_{n-2}$ by $b_{n-1}^{(1)}-(b_{n-1}^{(2)})^2$ to obtain $R_{n-1}$, which by definition is the ring $R$ in the statement of the theorem.  At each step, we are factoring by a regular element.   Taking the pullbacks we obtain a commutative diagram
\[
\begin{tikzpicture}
\node (A) at (0,0) {$\scriptstyle Y_{n-1}$};
\node (B) at (2,0) {$\scriptstyle Y_{n-2}$};
\node (C) at (4,0) {$\scriptstyle \hdots$}; 
\node (D) at (6,0) {$\scriptstyle Y_{1}$};
\node (E) at (8,0) {$\scriptstyle Y$};
\node (Aa) at (0,-1.5) {$\scriptstyle \Spec R$};
\node (Ba) at (2,-1.5) {$\scriptstyle \Spec R_{n-2}$};
\node (Ca) at (4,-1.5) {$\scriptstyle \hdots$}; 
\node (Da) at (6,-1.5) {$\scriptstyle \Spec R_1$};
\node (Ea) at (8,-1.5) {$\scriptstyle \Spec S$};
\draw[->] (A) -- node[above] {$\scriptstyle i_{n-1}$} (B);
\draw[->] (B) -- node[above] {$\scriptstyle i_{n-2}$}(C);
\draw[->] (C) -- node[above] {$\scriptstyle i_2$}(D);
\draw[->] (D) -- node[above] {$\scriptstyle i_1$}(E);
\draw[->] (Aa) -- node[above] {$\scriptstyle j_{n-1}$} (Ba);
\draw[->] (Ba) -- node[above] {$\scriptstyle j_{n-2}$} (Ca);
\draw[->] (Ca) -- node[above] {$\scriptstyle j_{2}$} (Da);
\draw[->] (Da) -- node[above] {$\scriptstyle j_{1}$} (Ea);
\draw[->] (A) -- node[left] {$\scriptstyle f_{n-1}$} (Aa);
\draw[->] (B) -- node[left] {$\scriptstyle f_{n-2}$} (Ba);
\draw[->] (D) -- node[left] {$\scriptstyle f_{1}$} (Da);
\draw[->] (E) -- node[left] {$\scriptstyle f$} (Ea);
\end{tikzpicture}
\]
By \cite{W}, under the setup above $\cv_{n-1}:=i_{n-1}^*\hdots i_1^*\cv$ is a tilting bundle on $Y_{n-1}$ with $\End_{Y_{n-1}}(\cv_{n-1})\cong j_{n-1}^*\hdots j_1^*\End_{S}(f_*\cv)\cong j_{n-1}^*\hdots j_1^*\End_S(S\oplus (a,b_1^{(2)}))$.  But on the other hand, $f_{n-1}$ is a projective birational morphism with fibres at most one-dimensional between integral normal schemes, and so 
\[
\End_{Y_{n-1}}(\cv_{n-1})\cong\End_{R}((f_{n-1})_*\cv_{n-1})\cong\End_{R}(j_{n-1}^*\hdots j_1^*f_*\cv)\cong\End_{R}(R\oplus (a,b_1)).
\]
where the middle isomorphism follows by iterating \cite[8.1]{IU}.  
Thus $\End_{R}(R\oplus (a,b_1))\cong j_{n-1}^*\hdots j_1^*\End_S(S\oplus (a,b_1^{(2)}))$.  Since by definition each $j_t^*$ factors by a regular element, we obtain $\End_R(R\oplus (a,b_1))$ from the presentation of $\End_S(S\oplus (a,b_1^{(2)}))$ in Step 3 by factoring out by the central relations corresponding to the regular elements.  Now, via the explicit form in Step 3, these are
\[
\begin{array}{rcl}
b_1^{(1)}-b_1^{(2)}&\leftrightarrow& (as_2+s_2a)-(bs_1+s_1b) \\
&\vdots&\\
b_{n-2}^{(1)}-b_{n-2}^{(2)}&\leftrightarrow&(as_{n-1}+s_{n-1}a)-(bs_{n-2}+s_{n-2}b) \\
b_{n-1}^{(1)}-(b_{n-1}^{(2)})^2&\leftrightarrow&(as_{n}+s_{n}a)-(bs_{n-1}+s_{n-1}b)^2 .
\end{array}
\]

\medskip
\noindent
{\em Step 5.} We justify that $\Lambda_n\cong\End_R(R\oplus (a,b_1))$.   From Step 4 we know that $\End_R(R\oplus (a,b_1))$ can be presented as
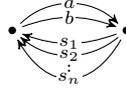
\[
\begin{array}{c}
\begin{tikzpicture} [bend angle=45, looseness=1]
\node (C1) at (0,0) [vertex] {};
\node (C2) at (1.5,0)  [vertex] {};
\draw [->,bend left] ($(C1)+(40:4pt)$) to node[gap]  {$\scriptstyle a$} ($(C2)+(140:4pt)$);
\draw [->,bend left=20,looseness=1] ($(C1)+(20:4pt)$) to node[gap]  {$\scriptstyle b$} ($(C2)+(160:4pt)$);
\draw [->,bend left=20] ($(C2)+(-160:4pt)$) to node[gap]  {$\scriptstyle s_1$} ($(C1)+(-20:4pt)$);
\draw [->,bend left=40] ($(C2)+(-140:4pt)$) to node[gap]  {$\scriptstyle s_2$} ($(C1)+(-40:4pt)$);
\draw [->,bend left=60,looseness=1.5] ($(C2)+(-120:4pt)$) to node[gap]  {$\scriptstyle s_n$} ($(C1)+(-60:4pt)$);
\node at (0.75,-0.47) {$\scriptstyle .$};
\node at (0.75,-0.53) {$\scriptstyle .$};
\end{tikzpicture}
\end{array}
\] 
subject to the relations
\[
\begin{array}{cl}
as_{i}b=bs_ia&\mbox{for all }1\leq i\leq n\\
s_ias_j=s_jas_i&\mbox{for all } 1\leq i<j\leq n.\\
as_n=(bs_{n-1})^2&\\
s_na=(s_{n-1}b)^2&\\
as_{i+1}=bs_i&\mbox{for all } 1\leq i\leq n-2\\ 
s_{i+1}a=s_ib&\mbox{for all } 1\leq i\leq n-2.
\end{array}
\]
This is a non-minimal presentation, since some relations can be deduced from others.  It is not difficult to show that the non-minimal presentation above can be reduced to the relations defining $\Lambda_n$.  This proves (2).

For the final statement in the theorem, by completing both sides we see that $\widehat{\Lambda}_n\cong\End_{\widehat{R}}(N^{\cc})$, which by \ref{D is derived equiv} is derived equivalent to the rational double point resolution $X^\cc$ of $\Spec\widehat{R}$. Since by construction $X^\cc$ has only one singularity, of type $\frac{1}{2}(1,1)$, $\underline{\GP}(\widehat{\Lambda}_n)\simeq\uCM(\mathbb{C}[[x,y]]^{\frac{1}{2}(1,1)})$ follows from \ref{main triangles}. Finally, since the center of $\widehat{\Lambda}_n$ is $\mathbb{C}[[x,y]]^{\frac{1}{2n-1}(1,2)}$, it follows that $n^\prime\neq n$ implies $\widehat{\Lambda}_{n^\prime}\ncong \widehat{\Lambda}_n$.
\end{proof}

\subsection{Frobenius structures on module categories}
Let $K$ be a field and denote $\mathbb{D}:=\Hom_K(-,K)$. Here we illustrate our main theorem \ref{t:main-thm} in the setting of finite dimensional algebras.  Using both \ref{t:main-thm} and \ref{new Frobenius structure}, we recover the following result due to 
Auslander--Solberg \cite{AS93-2}, which is rediscovered and generalised by Kong~\cite{Kong12}.

\begin{prop}\label{p:ASK}
Let $\Lambda$ be a finite dimensional algebra and $\cn$ a functorially finite subcategory of $\mod\Lambda$ satisfying $\Lambda\oplus \mathbb{D}\Lambda\in\cn$ and $\tau\underline{\cn}=\overline{\cn}$ where $\tau$ is the AR translation.
Then $\mod\Lambda$ has a structure of a Frobenius category such that the category of projective objects is $\add\cn$, and we have an equivalence $\mod\Lambda\to\GP( \cn),~X\mapsto \Hom_\Lambda(X,-)|_\cn$.
\end{prop}
\begin{proof}
By \ref{new Frobenius structure}, we have a new structure of a Frobenius category on $\mod\Lambda$ whose projective--injective objects are $\add\cn$.   Applying \ref{main-thm-category-version} to $(\ce,\cm,\cp):=(\mod\Lambda,\mod\Lambda,\add\cn)$, we have the assertion since $\mod(\mod\Lambda)$ has global dimension at most two and $\mod\Lambda$ is idempotent complete.
\end{proof}

The following result supplies a class of algebras satisfying the conditions in~\ref{p:ASK}. It generalises~\cite[3.4]{Kong12} in which $\Gamma$ is the path algebra of a Dynkin quiver. Below $\otimes:=\otimes_K$.
\begin{prop}\label{p:generalisation-of-K3.4}
Let $\Delta$ and $\Gamma$ be 
finite-dimensional $K$-algebras. Assume that $\Delta$ is selfinjective.
Then $\Lambda=\Delta \otimes \Gamma$ and $\cn= \Delta \otimes \mod \Gamma:=\{\Delta\otimes M\mid M\in\mod\Gamma\}$ satisfy the conditions in~\ref{p:ASK}.
Consequently, we have an equivalence
$$\mod \Lambda \cong \GP(\Delta \otimes \mod \Gamma).$$
\end{prop}
\begin{proof} 
Since $\Delta$ is selfinjective, both $\Lambda=\Delta\otimes\Gamma$ and $\mathbb{D}\Lambda=\mathbb{D}(\Delta\otimes\Gamma)=\mathbb{D}\Delta\otimes \mathbb{D}\Gamma=\Delta\otimes \mathbb{D}\Gamma$ belong to $\cn=\Delta\otimes\mod\Gamma$. For $M\in\mod\Gamma$, it follows from the next lemma that $\tau_\Lambda(\Delta\otimes M)=\nu_\Delta(\Delta)\otimes\tau_\Gamma(M)$. Since $\Delta$ is selfinjective, we have $\nu_\Delta(\Delta)=\Delta$, and hence $\tau_\Lambda(\Delta\otimes M)=\Delta\otimes\tau_\Gamma(M)\in \Delta\otimes\mod\Gamma$. Thus the conditions in~\ref{p:ASK} are satisfied. 
\end{proof}
\begin{lemma}
Let $\Delta$ and $\Gamma$ be finite-dimensional $K$-algebras and $\Lambda=\Delta\otimes\Gamma$. Then for a finite-dimensional $\Gamma$-module $M$ and a finitely generated projective $\Delta$-module $P$, we have $\tau_\Lambda(P\otimes M)=\nu_\Delta(P)\otimes \tau_\Gamma(M)$, where $\nu_\Delta=\mathbb{D}\Hom_\Delta(-,\Delta)$ is the Nakayama functor.
\end{lemma}
\begin{proof}
This is shown in the proof of~\cite[3.4]{Kong12} for the case when $\Delta$ is self-injective and $\Gamma$ is the path algebra of a Dynkin quiver. The proof there works more generally in our setting. For the convenience of the reader we include it here.

Let $Q^{-1}\stackrel{f}{\rightarrow}Q^0$ be a minimal projective presentation of $M$ over $\Gamma$. Then
\[
P\otimes Q^{-1}\xrightarrow{\mathrm{id}_P\otimes f} P\otimes Q^0
\]
is a minimal projective presentation of $P\otimes M$ over $\Delta\otimes\Gamma$. We apply $\nu_{\Lambda}=\mathbb{D}\Hom_{\Delta\otimes\Gamma}(-,\Delta\otimes\Gamma)$ and by the definition of $\tau$ we obtain an exact sequence
\begin{eqnarray}
0\to\tau_{\Lambda}(P\otimes M)\to\nu_{\Lambda}(P\otimes Q^{-1})\xrightarrow{\nu(\mathrm{id}_P\otimes f)}\nu_{\Lambda}(P\otimes Q^0).\label{e:tau}
\end{eqnarray}
Observe that for a finitely generated projective $\Gamma$-module $Q$ we have
\begin{eqnarray*}\nu_{\Lambda}(P\otimes Q)\hspace{-7pt}&=&\hspace{-7pt}\mathbb{D}\Hom_{\Delta\otimes\Gamma}(P\otimes Q,\Delta\otimes\Gamma)=\mathbb{D}(\Hom_\Delta(P,\Delta)\otimes\Hom_\Gamma(Q,\Gamma))\\
\hspace{-7pt}&=&\hspace{-7pt}\nu_\Delta(P)\otimes\nu_\Gamma(Q).\end{eqnarray*}
Therefore the sequence (\ref{e:tau}) is equivalent to
\[
0\to\tau_{\Lambda}(P\otimes M)\to\nu_{\Delta}(P)\otimes\nu_{\Gamma}(Q^{-1})\xrightarrow{\nu(\mathrm{id}_P)\otimes \nu(f)}\nu_{\Delta}(P)\otimes\nu_{\Gamma}(Q^0).
\]
It follows that $\tau_{\Lambda}(P\otimes M)=\nu_\Delta(P)\otimes\tau_\Gamma(M)$, as desired.
\end{proof}

\begin{remark}Let $\Delta$, $\Gamma$ and $\Lambda$ be as in~\ref{p:generalisation-of-K3.4}. Assume further that $\Gamma$ has finite representation type and let $\Aus(\Gamma)$ denote the Auslander algebra of $\Gamma$, i.e.\ the endomorphism algebra of an additive generator of $\mod \Gamma$.\\
\t{(1)}  The algebra $\Delta\otimes \Aus(\Gamma)$ is Iwanaga--Gorenstein and we have an equivalence
\[
\mod \Lambda \cong \GP(\Delta \otimes \Aus(\Gamma)).
\]
\t{(2)} If in addition $\mod\Gamma$ has no stable $\tau$-orbits, then any subcategory of $\Delta\otimes\mod\Gamma$ satisfying the conditions in~\ref{p:ASK} already additively generates $\Delta\otimes\mod\Gamma$. In this sense, $\Delta\otimes \Aus(\Gamma)$ is smallest possible.
\end{remark}

\subsection{Frobenius categories arising from preprojective algebras}

Let $Q$ be a finite quiver without oriented cycles and let $W$ be the Coxeter group associated to $Q$ with generators $s_i,~i\in Q_0$. Let $K$ be a field, let $\Lambda$ be the associated preprojective algebra over $K$ and let $e_i$ be the idempotent of $\Lambda$ corresponding to the vertex $i$ of $Q$. Denote $I_i=\Lambda(1-e_i)\Lambda$.

For an element $w\in W$ with reduced expression $w=s_{i_1}\cdots s_{i_k}$, let $I_w=I_{i_1}\cdots I_{i_k}$ and set $\Lambda_w=\Lambda/I_w$. As a concrete example, if $Q$ is the quiver of type $A_3$ and $w=s_2s_1s_3s_2$, then $\Lambda_w$ is given by the following quiver with relations

\[
\begin{array}{cc}
\begin{array}{c}
\begin{tikzpicture}[bend angle=15, looseness=1]
\node (a) at (-1,0) [vertex] {};
\node (b) at (0,0) [vertex] {};
\node (c) at (1,0) [vertex] {};
\node (a1) at (-1,-0.2) {$\scriptstyle 1$};
\node (b1) at (0,-0.2) {$\scriptstyle 2$};
\node (c1) at (1,-0.2) {$\scriptstyle 3$};
\draw[<-,bend right] (a) to node[below] {$\scriptstyle a^*$} (b);
\draw[->,bend left] (a) to node[above] {$\scriptstyle a$} (b);
\draw[<-,bend right] (b) to node[below] {$\scriptstyle b^*$} (c);
\draw[->,bend left] (b) to node[above] {$\scriptstyle b$} (c);
\end{tikzpicture}
\end{array}
&
{\small
\begin{array}{ccc}
aa^*=0 & b^*b=0 &a^*a=bb^*\\
ab=0 & b^*a^*=0.\\
\end{array}}
\end{array}
\]
Note that $I_w$ and $\Lambda_w$ do not depend on the choice of the reduced expression. By~\cite[III.2.2]{BIRSc}, $\Lambda_w$ is finite-dimensional and is Iwanaga--Gorenstein of dimension at most $1$. In this case, the category of Gorenstein projective $\Lambda_w$-modules coincides with the category $\Sub\Lambda_w$ of submodules of finitely generated projective $\Lambda_w$-modules. 
By~\cite[III.2.3, III.2.6]{BIRSc}, $\Sub\Lambda_w$ is a Hom-finite stably 2-Calabi--Yau Frobenius category and it admits a cluster-tilting object $M_w$. These results were stated in~\cite{BIRSc} only for non-Dynkin quivers, but they also hold for Dynkin quivers.

\newcommand{\nil}{\mathrm{nil~}}
Another family of Hom-finite stably 2-Calabi--Yau Frobenius categories with cluster-tilting object are constructed by Gei\ss, Leclerc and Schr\"oer in~\cite{GLS}. Precisely, for a terminal module $M$ over $KQ$ (i.e. $M$ is preinjective and $\add M$ is closed under taking the inverse Auslander--Reiten translation), consider $\cc_M=\pi^{-1}(\add M)\subseteq \nil\Lambda$, where $\nil\Lambda$ is the category of finite-dimensional nilpotent representations over $\Lambda$ and $\pi:\nil\Lambda\rightarrow\mod kQ$ is the restriction along the canonical embedding $KQ\rightarrow\Lambda$. Gei\ss, Leclerc and Schr\"oer show that $\cc_M$ admits the structure of a Frobenius category which is stably 2-Calabi--Yau with a cluster tilting object $T^\vee_M$. To $M$ is naturally associated an element $w$ of $W$. By comparing $T^\vee_M$ with $M_w$, they show that there is an anti-equivalence $\cc_M\rightarrow\Sub\Lambda_w$ \cite[\S22.7]{GLS}. 

We now explain how the results in this paper can be used to give a different proof of the equivalence $\cc_M\cong\Sub\Lambda_w^{\op}$.  

In~\cite[\S 8.1]{GLS}, an explicit construction of a projective generator $I_M$ of the Frobenius category $\cc_M$ is given. One can check that $\End_{\cc_M}(I_M)\cong \Lambda_w^{\op}$. By~\cite[13.6(2)]{GLS}, $\End_{\cc_M}(T^\vee_M)$ has global dimension $3$. Since $T^\vee_M$ has $I_M$ as a direct summand, it follows from \ref{t:main-thm} that 
\[
\cc_M\cong \GP(\Lambda_w^{\op}),
\]
and since $\id\Lambda_w^{\op}=1$, we have $\GP(\Lambda_w^{\op})=\Sub \Lambda_w^{\op}$.  Thus $\cc_M\cong\Sub \Lambda_w^{\op}$ follows.

\emph{Acknowledgements.} The first author was supported by JSPS Grant-in-Aid for Scientific Research 21740010, 21340003, 20244001 and 22224001. The first and third authors were partially supported by the EMS Research Support Fund. The second author was supported by the DFG grant Bu-1866/2-1.The fourth author was supported by the DFG
program SPP 1388 Ko-1281/9-1.

\end{document}